\documentclass[11pt]{article}
\usepackage{amssymb,amsmath,amsthm,pstricks}
\usepackage{graphics,graphicx}

\makeatletter\@addtoreset {equation}{section}\makeatother

\def\bexe{\begin{exercise}}\def\eexe{\eex\end{exercise}}
\def\bsol{\begin{solution}}\def\esol{\eex\end{solution}}
\def\bexa{\begin{example}}\def\eexa{\eex\end{example}}
\def\brem{\begin{remark}}\def\erem{\end{remark}}
\def\bthm{\begin{theorem}}\def\ethm{\end{theorem}}
\def\blem{\begin{lemma}}\def\elem{\end{lemma}}
\def\bcor{\begin{corollary}}\def\ecor{\end{corollary}}
\def\beq{\begin{equation}}\def\eeq{\end{equation}}
\def\bpf{\begin{proof}}\def\epf{\end{proof}}

\newcommand{\R}{{\mathbb R}}
\newcommand{\N}{{\mathbb N}}
\newcommand{\Z}{{\mathbb Z}}

\newcommand{\bi}{\begin{itemize}}\newcommand{\ei}{\end{itemize}}
\newcommand{\ben}{\begin{enumerate}}\newcommand{\een}{\end{enumerate}}
\newcommand{\bce}{\begin{center}}\newcommand{\ece}{\end{center}}

\newcommand{\barr}{\begin{array}}\newcommand{\earr}{\end{array}}
\newcommand{\bpm}{\begin{pmatrix}}\newcommand{\epm}{\end{pmatrix}}
\newcommand{\ba}{\begin{array}}\newcommand{\ea}{\end{array}}

\def\eex{\hfill\mbox{$\rfloor$}}

\def\bd{\begin{displaymath}} \def\ed{\end{displaymath}}
\def\ba{\begin{array}} \def\ea{\end{array}}

\newtheorem{theorem}{Theorem}[section]

\newtheorem{lemma}[theorem]{Lemma}

\newtheorem{exercise}[theorem]{Exercise}
\newtheorem{solution}[theorem]{Solution}
\newtheorem{remark}[theorem]{Remark}
\newtheorem{corollary}[theorem]{Corollary}
\newtheorem{example}[theorem]{Example}

\begin{document}

\title{\bf Localized Modes of the Linear Periodic Schr\"{o}dinger Operator with a Nonlocal Perturbation}



\author{Tom\'{a}\v{s} Dohnal$^1$, Michael Plum$^2$ and Wolfgang Reichel$^2$\\
{\small $^1$ Institut f\"ur Angewandte und Numerische Mathematik}\\
{\small $^2$ Institut f\"ur Analysis}\\
{\small Fakult\"{a}t f\"{u}r Mathematik, Universit\"{a}t Karlsruhe (TH), Germany}}

\date{\today}
\maketitle


\begin{abstract}
We consider the existence of localized modes corresponding to eigenvalues of the periodic Schr\"{o}dinger operator $-\partial_x^2+ V(x)$ with an interface. The interface is modeled by a jump either in the value or the derivative of $V(x)$ and, in general, does not correspond to a localized perturbation of the perfectly periodic operator. The periodic potentials on each side of the interface can, moreover, be different. As we show, eigenvalues can only occur in spectral gaps. We pose the eigenvalue problem as a $C^1$ gluing problem for the fundamental solutions (Bloch functions) of the second order ODEs on each side of the interface. The problem is thus reduced to finding matchings of the ratio functions $R_\pm=\frac{\psi_\pm'(0)}{\psi_\pm(0)}$, where $\psi_\pm$ are those Bloch functions that decay on the respective half-lines. These ratio functions are analyzed with the help of the Pr\"{u}fer transformation. The limit values of $R_\pm$ at band edges depend on the ordering of Dirichlet and Neumann eigenvalues at gap edges. We show that the ordering can be determined in the first two gaps via variational analysis for potentials satisfying certain monotonicity conditions. Numerical computations of interface eigenvalues are presented to corroborate the analysis.
\end{abstract}

\section{Introduction}\label{S:intro}

Localization for perturbed periodic Schr\"{o}dinger operators $L=-\Delta + V_0(x) + \tilde{V}(x)$, where $V_0(x)$ is periodic in $x\in \R^n, n\in \N$, is a classical problem traditionally treated by spectral theory. Most commonly it is studied for perturbations $\tilde{V}(x)$ that are either compactly supported, see, e.g., Deift \& Hempel \cite{DH86}, Alama et al. \cite{ADH89}, and Borisov \& Gadyl'shin \cite{BG07} or fast decaying, e.g. $\tilde{V}\in L^{n/2}(\R^n)$, cf. \v{Z}eludev \cite{Zelud67} and Alama et al. \cite{ADH89}. 
Both of these scenarios can lead to eigenvalues of $L$ and thus to localization. Potentials $\tilde{V}$ describing random perturbation also yield eigenvalues due to Anderson localization, studied, for example, by Kirsch et al. \cite{KSS98} and Veseli\'{c} \cite{V02}. We investigate localization in the one-dimensional case $n=1$ due to the presence of deterministic interfaces which cannot be represented as localized perturbations of $-\partial_x^2+V_0(x)$. Such an interface arises, for instance, when $\tilde{V}(x)$ is periodic on one side of the interface and vanishes on the other side (we assume commensurability of the periods of $\tilde{V}$ and $V_0$ to preserve periodicity on each side of the interface). This topic has been previously studied mainly by Korotyaev via spectral theory \cite{KOR00,KOR05}. We, on the other hand, use the properties of the fundamental solutions of the 1D spectral problems of the periodic operators corresponding to each side of the interface and pose the eigenvalue problem as a $C^1$-gluing problem for the decaying Floquet-Bloch solutions from either interface side.
 This approach allows us to provide some concrete conditions on $V_0$ and the perturbation $\tilde{V}$ directly (without conditions on the spectrum of $-\partial_x^2 + V_0(x)$) that ensure eigenvalue existence in the semi-infinite and the first finite gap of the continuous spectrum of $L$. Our approach is also arguably conceptually simpler than that of \cite{KOR00,KOR05}.

Localized waves at interfaces of two periodic (linear) structures have been also demonstrated experimentally in the context of electron waves in crystals by Ohno et al. \cite{OMAH92} and for optical waves in photonic crystals by, e.g., Suntsov et al. \cite{Sunts_etal07}.

In detail, within the framework of the eigenvalue problem
\beq\label{E:L_prob}
L\psi = \lambda \psi, \qquad L = -\partial_x^2 + V(x), \ x\in \R
\eeq
we study the following two interface problems. Firstly, an \textit{interface made of even periodic potentials}
\beq\label{E:even_interf}
V(x) = \chi_{\{x<0\}}V_-(x)+\chi_{\{x\geq 0\}}V_+(x), 
\eeq
where $V_\pm$ has period $d_\pm>0$, i.e., $V_\pm(x+d_\pm)=V_\pm(x)$ for all $x\in \R$, and furthermore satisfies $V_\pm(\frac{d_\pm}{2}+x)=V_\pm(\frac{d_\pm}{2}-x)$ for all $x\in[0,\frac{d_\pm}{2}]$. Secondly, an \textit{interface made of dislocated even periodic potentials}
\beq\label{E:transl_interf}
V(x) = \chi_{\{x< 0\}}V_0(x+s)+ \chi_{\{x\geq 0\}} V_0(x+t),
\eeq
where $V_0$ has period $d>0$, i.e., $V_0(x+d)=V_0(x)$ for all $x\in \R$, and satisfies $V_0(\frac{d}{2}+x)=V_0(\frac{d}{2}-x)$ for all $x\in[0,\frac{d}{2}]$. The dislocation parameters are $t,s \in \R$.
Here $\chi$ is the characteristic function. Note that under the periodicity conditions the evenness of $V_\pm$ and $V_0$ about $x=\frac{d_\pm}{2}$ and $x=\frac{d}{2}$ within the periodicity cell $[0,d_\pm]$ and $[0,d]$, respectively, is equivalent to evenness of $V_\pm$ and $V_0$ about $x=0$. Hence, in the following we will simply require that the potentials be periodic and even.
Unless otherwise stated, the potentials $V_\pm$ and $V_0$ are continuous and hence bounded. 

One of the simplest examples of the interface \eqref{E:even_interf} is the \textit{additive interface}
\beq\label{E:add_interf}
V_-(x) = V_0(x), \ V_+(x) = V_0(x)+\alpha, \quad V_0(x+d)=V_0(x),  \ V_0(-x)=V_0(x), \quad \alpha \in \R, \ d>0,
\eeq
generated by merely changing the average value of the potential on one half of the real axis. This example is studied in more detail in Section \ref{S:example_add_jump} since the conditions on eigenvalue existence become rather specific in this particular case.

Schematic pictures of the two potentials \eqref{E:even_interf} and \eqref{E:transl_interf} are displayed in Figure~\ref{F:V}.
\begin{figure}[!ht]
\begin{center}
\includegraphics[height=3cm,width=8cm]{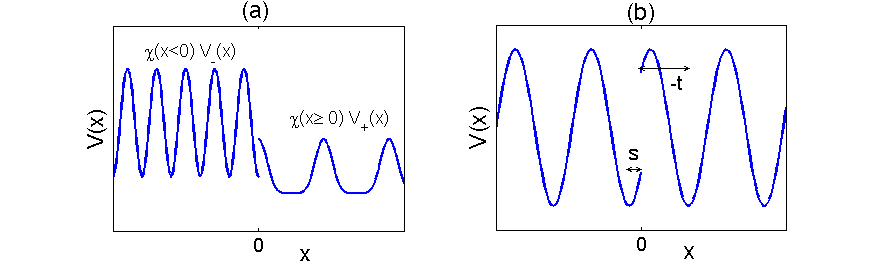}
\end{center}
\caption{A cartoon of example potentials $V$ for the case \eqref{E:even_interf} in (a) and \eqref{E:transl_interf} in (b).}
\label{F:V}
\end{figure}

Equation \eqref{E:L_prob} finds applications in many fields of natural science. Perhaps most notably it describes the wave function of an electron in a one dimensional crystal, where waves localized at a crystal interface are typically called Tamm states \cite{Tamm32}. The equation also directly applies to the description of light propagating transversally to the direction of periodicity of a non-dispersive, lossless, linear photonic crystal which is homogeneous in the $y$ and $z$ directions. Suppose the refractive index $n$ varies periodically in the $x-$direction and its mean has a jump at $x=0$, such that $n(x)=\sqrt{1+W(x)}, W(x) = -\frac{c^2}{\omega^2}V(x)$. We assume the following form of the electric field, 
$$\vec{E} = (0,\psi(x),0)^T e^{i(kz-\omega t)},$$
such that the field is polarized in the $y-$direction, the waves propagate in the $z-$ direction and the $x-$profile is stationary. Then Maxwell's equations exactly reduce to 
$$(\partial_x^2-k^2)\psi +\frac{\omega^2}{c^2}(1+W(x))\psi = 0.$$ 
With $V(x) = -\frac{\omega^2}{c^2}W(x)$ and setting $\lambda = \frac{\omega^2}{c^2}-k^2$, we recover \eqref{E:L_prob}.

Another example of an application of \eqref{E:L_prob} is the description of matter waves in one dimensional Bose-Einstein condensates loaded onto an optical lattice, see Choi \& Niu \cite{CN99}. The density of a condensate is described by the wavefunction $u$ governed by the Gross-Pitaevskii equation \cite{CN99,Gross63,Pitaev61}
$$ i \hbar \partial_t u + \frac{\hbar^2}{2m}\partial_x^2u -W(x) u -g|u|^2u=0,$$
where, in our setting, $W(x)$ is periodic but has a jump at $x=0$. Here $\hbar$  is Planck's constant, $m$ is the boson mass, $W$ is the potential induced by the optical lattice and $g$ is the scattering length. In the linear regime, $g=0$, stationary waves $e^{-i\lambda t}\psi(x)$ obey (after rescaling) equation \eqref{E:L_prob}.

The rest of the paper is organized as follows. In Section \ref{S:no_jump} we review the needed facts on spectral properties of the interface-free periodic Schr\"{o}dinger operators with an even potential including the problem of ordering of spectral band edges according to even/odd symmetry of the Bloch functions. Section \ref{S:jump_even_pot} discusses the interface \eqref{E:even_interf} and introduces the main tools of our analysis, namely the $C^1$-matching condition and the Pr\"{u}fer transformation. The theory is then applied to the additive interface example \eqref{E:add_interf} and numerical computations of point spectrum are performed. In Section~\ref{S:disloc_interf} we analyze the dislocation problem \eqref{E:transl_interf} for the cases $s=-t$ and $s=0$ using the same tools as in Section \ref{S:jump_even_pot} plus differential inequalities and variational methods. Numerical examples are, once again, provided.

\section{Spectrum of the Interface-Free Problem}\label{S:no_jump}
We review, first, some well known results on the spectrum and the eigenfunctions of the interface-free operator $L_0 := -\partial_x^2+V_0(x)$, where $V_0(x+d)=V_0(x)$ is continuous and $V_0(-x)=V_0(x)$. Good sources on the theory of the periodic Schr\"{o}dinger operator are Magnus and Winkler \cite{MagWin66}, Eastham \cite{Eastham} and Reed \& Simon \cite{RS}. 

$L_0$ has a purely continuous spectrum (see Theorem XIII.90 in \cite{RS}) consisting of bands $[s_{2n-1},s_{2n}]$ so that 
\[\sigma(L_0) = \bigcup_{n\in \N}[s_{2n-1},s_{2n}],\]
where $s_n\in\R$ and $s_{2n-1}<s_{2n}\leq s_{2n+1}$ \cite{Eastham}. When $s_{2n+1}>s_{2n}$, we say that $\sigma(L_0)$ has the finite gap $G_n:=(s_{2n},s_{2n+1})$. Clearly, $\sigma(L_0)$ has also the semi-infinite gap $G_0=(s_0,s_1):=(-\infty,s_1)$. According to Floquet theory \cite{Eastham} the spectrum $\sigma(L_0)$ can be easily found via the use of the monodromy matrix of the second order ODE $L_0\psi=\lambda \psi$. Figure~\ref{F:spec_L0} presents the numerically computed spectrum of the operator $L_0$ with  $V_0(x) = \sin^2(\pi x/10)$.
\begin{figure}[!ht]
\begin{center}
\includegraphics[scale=.65]{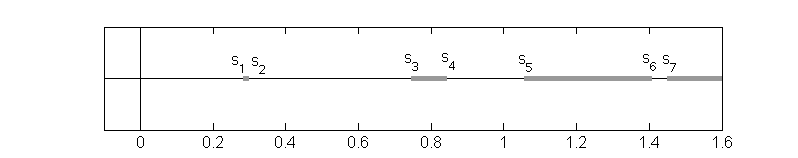}
\end{center}
\caption{Spectrum of $L_0$ for $V_0(x) = \sin^2(\pi x/10)$.}
\label{F:spec_L0}
\end{figure}

The ODE $L_0 \psi = \lambda \psi$ has two linearly independent solutions, so called, Bloch functions. For real $\lambda \not\in \partial \sigma(L_0)$ they are of the form
\beq\label{E:Bloch_form}
\psi_1(x;\lambda)=p_1(x;\lambda) e^{-ik(\lambda)x}, \quad \psi_2(x;\lambda)=p_2(x;\lambda)e^{ik(\lambda)x},
\eeq
where $k\in \R\setminus \{0\}$ if $\lambda \in \text{int}(\sigma(L_0))$ and $k\in i\R\setminus\{0\}$ if $\lambda \in \R \setminus \sigma(L_0)$, and $p_{1,2}(x;\lambda)$ are real-valued and $2d-$periodic in $x$. In fact, $p_{1,2}$ are either $d-$periodic of $d-$anti-periodic. If $\lambda \in \partial (\sigma(L_0))$, the Bloch functions are of the form 
\beq\label{E:Bloch_form_bdry}
\psi_1(x;\lambda)=p_1(x;\lambda), \quad \psi_2(x;\lambda)=p_2(x;\lambda) +x p_1(x;\lambda),
\eeq 
where again $p_{1,2}(x;\lambda)$ are real and $2d-$periodic in $x$. 


The evenness of the potential $V_0(x)$ and the fact that only one linearly independent bounded Bloch function (namely $\psi_1(x;\lambda) = p_1(x;\lambda)$) exists at any $\lambda \in \partial(\sigma(L_0))$ imply that this solution must be even or odd and hence it satisfies at the boundary-points $x=0$ and $x=d$ either Dirichlet- or Neumann-boundary conditions. For $k\in \N$ let $(\mu_k,\zeta_k)$ denote the $k$-th Dirichlet eigenpair of $L_0$ on $[0,d]$ satisfying $\zeta_k(0) = \zeta_k(d)=0$ and let $(\nu_k,\eta_k)$ be the $k$-th Neumann eigenpair of $L_0$ on $[0,d]$ such that $\eta'_k(0) = \eta'_k(d)=0$. The following lemma may be well known, cf. \cite{Eastham}, Theorem 1.3.4.

\begin{lemma} For the first gap edge we have $s_1=\nu_1$. If $k \geq 1$ and if $s_{2k}\not = s_{2k+1}$ then $s_{2k}=\min\{\mu_k,\nu_{k+1}\}$, $s_{2k+1}=\max\{\mu_k,\nu_{k+1}\}$. Moreover, the following properties of the eigenfunctions are known (note that the even/odd-property applies with respect to reflection about $\frac{d}{2}$):
\begin{table}[!ht]
\begin{center}
\begin{tabular}{|c|l|c|c|r|}
\hline
 & eigenvalue & \multicolumn{3}{|c|}{eigenfunction properties}\\
\hline
Dirichlet & $\mu_{2k-1}$ & even & $d-$anti-periodic & $\zeta_{2k-1}'(\frac{d}{2})=0$\\
\hline
Dirichlet & $\mu_{2k}$ & odd & $d-$periodic & $\zeta_{2k}(\frac{d}{2})=0$\\
\hline
Neumann & $\nu_{2k-1}$ & even & $d-$periodic & $\eta_{2k-1}'(\frac{d}{2})=0$\\
\hline
Neumann & $\nu_{2k}$ & odd & $d-$anti-periodic & $\eta_{2k}(\frac{d}{2})=0$\\
\hline
\end{tabular}
\end{center}
\end{table}
\label{d_half}
\end{lemma}
\noindent
{\em Remark.} Note that $\lambda\in G_n$ can never be a Dirichlet or Neumann eigenvalue since any corresponding eigenfunction could be extended to a bounded solution of $L_0 \psi=\lambda \psi$ on $\R$ by reflection and periodic extension. Such nontrivial solutions cannot exist for $\lambda\in G_n$ by \eqref{E:Bloch_form}.

As we show in Sections \ref{S:jump_even_pot} and \ref{S:disloc_interf}, ordering between the Dirichlet and Neumann eigenvalues $\mu_k$ and $\nu_{k+1}$ plays an important role for existence of interface eigenvalues. It is, however, known that all orderings are in general possible, i.e., for any given ordering of the Dirichlet and Neumann eigenvalues (respecting the condition $\max\{\mu_k,\nu_{k+1}\}\leq \min \{\mu_{k+1},\nu_{k+2}\}, k\in \N$) a corresponding even potential $V_0$ exists, see Theorem 3 in Garnett \& Trubowitz \cite{GT84}. Nevertheless, the following lemma provides an ordering of low eigenvalues under some monotonicity assumptions on the potential $V_0$.

\begin{lemma} \label{L:ordering}

\begin{itemize}
\item[(a)] If $V_0$ is strictly increasing on $[0,\frac{d}{2}]$, then $\nu_2<\mu_1$. The Neumann eigenfunction corresponding to $\nu_2$ is strictly monotone on $[0,d]$ and odd with respect to $\frac{d}{2}$. 
\item[(b)] If $V_0$ is strictly decreasing on $[0,\frac{d}{2}]$, then $\mu_1<\nu_2$. The Dirichlet eigenfunction corresponding to $\mu_1$ is strictly monotone on $[0,\frac{d}{2}]$ and even with respect to $\frac{d}{2}$. 
\end{itemize}
\end{lemma} 

The proof is based on the following result.

\begin{lemma} Consider a potential $V_0\in L^\infty(a,b)$ (not necessarily periodic, even or continuous) and let $\kappa_{ND}$ be the first eigenvalue of $L_0=-\partial_x^2+V_0$ on $[a,b]$ with the boundary condition $u'(a)=0=u(b)$, whereas $\kappa_{DN}$ denotes the first eigenvalue of the same differential operator but with boundary conditions $u(a)=0=u'(b)$. Then
\begin{equation}
\min\{\kappa_{ND},\kappa_{DN}\} = \min\left\{ \int_a^b {v'}^2+V_0(x) v^2\,dx: v \in H^1(a,b) \mbox{ has a zero and } \int_a^b v^2\,dx=1\right\}.
\label{min}
\end{equation}
Moreover, if $V_0$ is strictly increasing on $[a,b]$ then $\kappa_{ND}<\kappa_{DN}$ and any eigenfunction for $\kappa_{ND}$ with $u(a)>0$ is strictly decreasing on $[a,b]$. If $V_0$ is strictly decreasing on $[a,b]$ then $\kappa_{DN}<\kappa_{ND}$ and any eigenfunction for $\kappa_{DN}$ with $u'(a)>0$ is strictly increasing on $[a,b]$.
\label{min_lemma}
\end{lemma}

\bpf The proof is inspired by a similar result in Bandle et al. \cite{BBR}. Note first that the set, on which the minimization is performed, is weakly closed in $H^1(a,b)$ due to the compact embedding $H^1(a,b)\to C[a,b]$. Hence a minimizer of the right-hand side of \eqref{min} exists. We denote it by $U$. Let us also denote the value of the minimum by $\kappa$. The proof is now divided into five steps:

\smallskip

\noindent
{\em Step 1: $U$ has exactly one zero on $[a,b]$.} Since $U$ possesses at least one zero $x_0\in [a,b]$, we have $U\in H_{x_0}=\{v\in H^1(a,b): v(x_0)=0\}$. Clearly $U$ is then the minimizer of 
$$
\min\left\{ \int_a^b {v'}^2+V_0(x) v^2\,dx: v \in H_{x_0}, \int_a^b v^2\,dx=1\right\},
$$
and therefore $U$ satisfies the Euler-Lagrange equation
\begin{equation}
-U'' + V_0(x) U = \kappa U \mbox{ in } (a,x_0)\cup (x_0,b)
\label{el}
\end{equation}
with boundary condition 
\begin{equation}
U'(a)=U(x_0)=U'(b)=0
\label{bc}
\end{equation}
where in case $x_0\in \{a,b\}$ one of the two Neumann conditions is dropped. Note that 
\begin{equation}
\int_a^b U'v' + V_0(x)Uv\,dx = \kappa \int_a^b Uv\,dx \quad \mbox{ for all } v\in H_{x_0}.
\label{weak_form}
\end{equation}
Now assume for contradiction that $U$ has a second zero $x_1\not = x_0$. Then \eqref{weak_form} holds also for all $v\in H_{x_1}$ and since $H^1(a,b)=H_{x_1}\oplus H_{x_2}$, we find that \eqref{weak_form} holds for all $v\in H^1(a,b)$, i.e., $U$ is a Neumann-eigenfunction. The same applies for $|U|$, which is also a minimizer of \eqref{min}. But then $U$ must be the first Neumann-eigenfunction of $L_0$ on $(a,b)$ and it therefore has no zero on $[a,b]$. This contradiction shows that $U$ has exactly one zero in $[a,b]$. 

\smallskip

\noindent
{\em Step 2: $\kappa$ is strictly less than the second Neumann-eigenvalue $\nu_2$ on $[a,b]$.} Since the second Neumann eigenfunction $\eta_2$ has one zero in $[a,b]$, we find $\kappa\leq \nu_2$. Suppose for contradiction that $\kappa=\nu_2$. Testing the equation for $\eta_2$ with $\eta_2^+=\max\{\eta_2,0\}$ we obtain 
$$
\int_a^b ({\eta_2^+}')^2+ V_0(x)(\eta_2^+)^2\,dx = \nu_2 \int_a^b (\eta_2^+)^2\,dx
$$
and thus $\eta_2^+$ is a minimizer for \eqref{min} and must have a unique zero by Step 1. However, clearly $\eta_2^+$ has a continuum of zeros. Therefore we can conclude that $\kappa<\nu_2$.

\smallskip

\noindent
{\em Step 3: $U$ has its unique zero either at $x=a$ or at $x=b$.} If we suppose for contradiction that the unique zero $x_0$ lies in the open interval $(a,b)$, then we obtain the Euler-Lagrange equation \eqref{el} with boundary condition \eqref{bc}. By rescaling the minimizer $U$ suitably on $[a,x_0]$ we can achieve that the rescaled function $U$ is a $C^1$-function on $[a,b]$ solving the equation \eqref{el} pointwise a.e. on $(a,b)$. Hence, the rescaled function $U$ is a Neumann-eigenfunction with one interior zero, i.e., $\kappa=\nu_2$ in contradiction to Step 2.

\smallskip

Now the claim of the lemma about the value of the minimum is immediate.

\smallskip

\noindent
{\em Step 4: ordering of $\kappa_{ND}, \kappa_{DN}$.} We are using the following rearrangement result of Hardy, Littlewood, P\'{o}lya~\cite{HLP}. Let $v,w$ be non-negative and measurable on $[a,b]$. If $v^\sharp, w^\sharp$ are the increasing rearrangements of $v,w$, then $\int_a^b vw\,dx \leq \int_a^b v^\sharp w^\sharp\,dx$. Moreover, if $v$ is strictly increasing, then equality holds if and only if $w=w^\sharp$. A similar statement holds for the decreasing rearrangements $v^\ast, w^\ast$. Note, that the non-negativity of $v,w$ can be replaced by boundedness.

A simple corollary of the Hardy, Littlewood, P\'{o}lya inequality is the following: suppose 
$V=V^\sharp$ is strictly increasing and both $V$ and $w$ are bounded. Then 
\beq
\int_a^b Vw^\ast \,dx \leq \int_a^b Vw\,dx
\label{hlp_ineq}
\eeq
with equality if and only if $w=w^\ast$. The proof follows immediately from the observation that $(-w)^\sharp= -w^\ast$. 

Let $V_0$ be strictly increasing on $[a,b]$. Suppose for contradiction that $\kappa_{DN}\leq \kappa_{ND}$ and let $U$ be an eigenfunction corresponding to $\kappa_{DN}$, which by \eqref{min} is also a minimizer of the variational problem in \eqref{min}. We may assume $U$ to be non-negative, since $|U|$ is also a minimizer of the corresponding variational problem and $\kappa_{DN}$ is a simple eigenvalue. Let now $U^\ast$ be the decreasing rearrangement of $U$ on $[a,b]$ and note that $(U^2)^\ast=(U^\ast)^2$. Since for the decreasing rearrangement we have $\int_a^b(U^{*'})^2dx\leq \int_a^b(U')^2dx$, cf. Kawohl \cite{KAW}, we obtain by \eqref{hlp_ineq} applied to $V_0$ and $U^2$ the relations 
\begin{equation}
\int_a^b (U^\ast)^2\,dx = \int_a^b U^2\,dx=1, \qquad \int_a^b ({U^\ast}')^2+ V_0(x) (U^\ast)^2\,dx 
\leq \int_a^b (U')^2+ V_0(x) U^2\,dx.
\label{re}
\end{equation}
Therefore $U^\ast$, which satisfies $U^\ast(b)=0$, is also a minimizer of \eqref{min} and hence equality has to hold in \eqref{re}.  But since $V_0$ is strictly increasing, the sharp form of \eqref{hlp_ineq} implies that $U=U^\ast$ which by $U(a)=0$ implies the contradiction that $U$ must be identically zero. Hence $\kappa_{ND}<\kappa_{DN}$. Moreover, \eqref{re} shows that any non-negative minimizer $U$ for $\kappa_{ND}$ satisfies $U=U^\ast$, i.e., $U$ is decreasing, and by using the differential equation for $U$ and the strict monotonicity of $V_0$ it is easy to see that in fact $U$ is strictly decreasing.

If $V_0$ is strictly decreasing on $[a,b]$ then a similar argument based on replacing $U$ by its increasing rearrangement shows that $\kappa_{DN}<\kappa_{ND}$. 
\epf

\begin{proof}[Proof of Lemma \ref{L:ordering}] Consider the Dirichlet-eigenfunction $\zeta_1$. By Lemma \ref{d_half} its restriction to $[0,\frac{d}{2}]$ is the eigenfunction for $\kappa_{DN}$ of Lemma \ref{min_lemma}. Likewise, the restriction of $\eta_2$ to $[0,\frac{d}{2}]$ is the eigenfunction for $\kappa_{ND}$. Hence $\mu_1=\kappa_{DN}$ and $\nu_2=\kappa_{ND}$. The statements (a) and (b) then follow from Lemma~\ref{min_lemma}.
\end{proof}

\section{Interface Problems}\label{S:jump}

Let $L$ be the operator in \eqref{E:L_prob} defined on the dense subset $H^2(\R)$ of $L^2(\R)$. We investigate next the existence of eigenvalues of $L$ for the interface potentials \eqref{E:even_interf} and \eqref{E:transl_interf}. These examples fall into a larger class of potentials, namely $V(x)=\chi_{\{x<0\}}V_1(x)+\chi_{\{x\geq 0\}}V_2(x)$, where $V_{1,2}(x+d_{1,2})=V_{1,2}(x)$ for some $d_{1,2}\geq 0$ but where $V_{1,2}$ may not be even in $x$. Clearly, all solutions of $(-\partial_x^2+V(x))\psi=\lambda \psi$ are then
\[\psi(x)=\chi_{\{x<0\}}\psi_-(x)+\chi_{\{x\geq 0\}}\psi_+(x),\]
where $\psi_\pm$ are Bloch functions of $(-\partial_x^2+V_{1,2}(x))\psi=\lambda \psi$, respectively. As decaying Bloch functions $\psi_{\pm}$ exist only in spectral gaps of $-\partial_x^2+V_{1,2}(x)$, respectively, eigenvalues of $L$ can exist only within intersections of the gaps of $\sigma(-\partial_x^2+V_{1}(x))$ and $\sigma(-\partial_x^2+V_{2}(x))$. Note the following additional information on the spectrum of $L$, which for our purpose plays no further role: the essential spectrum of $L$ is the union of the essential spectra of $-\partial_x^2+V_1(x)$ and $-\partial_x^2+V_2(x)$, cf. Korotyaev \cite{KOR05}. As a result, no embedded eigenvalues of $L$ exist.

\subsection{Point Spectrum for Interfaces Made of Even Potentials}\label{S:jump_even_pot}

The eigenvalue problem \eqref{E:L_prob} with \eqref{E:even_interf} can be viewed as the system
\beq\label{E:L_system}
\begin{array}{ll}
L_- \psi := -\partial_x^2 \psi + V_-(x) \psi  = \lambda \psi & \text{for} \ x < 0, \\
L_+ \psi := -\partial_x^2 \psi + V_+(x) \psi = \lambda \psi & \text{for} \ x \geq 0 \\
\end{array}
\eeq
coupled by the $C^1$-matching conditions 
\beq\label{E:match}
\psi(0-)=\psi(0+) \qquad \text{and} \qquad  \psi'(0-) = \psi'(0+).
\eeq
As stated in Section \ref{S:intro}, the functions $V_\pm(x)$ are continuous, even and $d_\pm$-periodic.

Based on the knowledge of the fundamental solutions in \eqref{E:Bloch_form}, \eqref{E:Bloch_form_bdry} we conclude that an $L^2$-integrable solution of \eqref{E:L_prob} with \eqref{E:even_interf} can only exist if $\lambda$ lies in the intersection of the resolvent sets, i.e., in the intersection of the spectral gaps of $L_-$ and $L_+$, i.e., if $\lambda \in G_n^+\cap G_m^-$ for some $n,m \in \N\cup \{0\}$, where $G_n^\pm$ is the $n$-th spectral gap of $L_\pm$ respectively. 

For $\lambda\in G_n^+\cap G_m^-$ with some $n,m\geq 0$ any localized eigenfunction $\psi$ of $L$, therefore, has to be of the form
\[\psi(x;\lambda)  = \chi_{\{x<0\}}\psi_-(x;\lambda) + \chi_{\{x\geq 0\}}\psi_+(x;\lambda), \]
where
\beq\label{E:decaying_parts}
\psi_\pm(x;\lambda) = p_\pm(x;\lambda)e^{\mp\kappa(\lambda)x}
\eeq
with $\kappa(\lambda)>0$ and $p_\pm(x;\lambda)$ being $2d_\pm-$periodic in $x$. The functions $p_\pm$ are restrictions of either $p_1$ or $p_2$ in \eqref{E:Bloch_form} with $V_0=V_\pm$ to the half-line $\R_\pm$ respectively.

An important remark is that, due to the linearity of the problem, the matching conditions \eqref{E:match} together with an appropriate scaling are equivalent to
\beq\label{E:R_cond}
R_+(\lambda)=R_-(\lambda), \quad\mbox{ where } \quad R_\pm(\lambda) = \frac{\psi'_{\pm}(0;\lambda)}{\psi_{\pm}(0;\lambda)}
\eeq
and the prime denotes differentiation in $x$.

We determine existence of solutions to \eqref{E:R_cond} via the intermediate value theorem and by monotonicity of the functions $R_\pm(\lambda)$. The monotonicity then also implies uniqueness.

\blem\label{L:G_monotone}
Within each gap $G_n^+$ and $G_n^-, n\geq 0$, the functions $R_+$ and $R_-$ are continuous functions of $\lambda\in G_n^\pm$, which are strictly increasing and decreasing respectively.
\elem
\bpf \ Let us start with the proof for $R_+(\lambda)$. Under the Pr\"{u}fer transformation, cf. Coddington \& Levinson \cite{CL}  
$$\psi_+(x;\lambda) = \rho(x;\lambda)\sin(\theta(x;\lambda)), \qquad \psi'_+(x;\lambda) = \rho(x;\lambda)\cos(\theta(x;\lambda)),$$ 
the equation $ L_+\psi_+=\lambda \psi_+$ becomes
\[
\begin{split}
\theta' & = 1+(\lambda - V_+(x)-1)\sin^2(\theta),\\
\rho' & = -\rho(\lambda - V_+(x)-1)\sin(\theta)\cos(\theta),
\end{split}
\]
where the prime denotes differentiation in $x$. Clearly, $\theta$ and $\rho$ are continuous functions of both variables $x\in \R$ and $\lambda\in G_n^+$ and since $R_+(\lambda)=\cot(\theta(0;\lambda))$, the function $R_+(\lambda)$ is continuous in $\lambda$ provided $\psi_+(0;\lambda)$ has no zero in the interior of $G_n^+$. Note that if $\psi_+(0;\lambda)=0$, then by evenness of $V_+$ and the reflection symmetry of the problem $L_+\psi_+ = \lambda \psi_+$, the solution $\psi_+(x;\lambda)$ defined in \eqref{E:decaying_parts} on $x\geq 0$ could be extended to a solution on $x\in \R$ via $\psi_+(-x;\lambda)=-\psi_+(x;\lambda)$. This solution would decay exponentially at both infinities and $\lambda$ would, thus, be an eigenvalue of $L_+$, which is impossible. Hence continuity of $R_+(\lambda)$ is proven.

Now let us prove the monotonicity. Due to the form of $\psi_+$, see \eqref{E:decaying_parts}, we have
\beq\label{E:rho_2d}
\rho(2d_+) = \sqrt{(\psi_+(2d_+))^2+(\psi'_+(2d_+))^2} = e^{-2d_+\kappa}\rho(0).
\eeq
Define now $z(x):=\frac{\partial \theta}{\partial \lambda}(x;\lambda)$. The function $z$ satisfies $z' = z(\lambda-V_+(x)-1)2\sin(\theta)\cos(\theta) +\sin^2(\theta)=-2\frac{\rho'}{\rho}z+\sin^2(\theta)$. Therefore, 
\beq\label{E:z_var_const}
z(x) = \left(\frac{\rho(0;\lambda)}{\rho(x;\lambda)}\right)^2z(0) +\int_0^x \left(\frac{\rho(t;\lambda)}{\rho(x;\lambda)}\right)^2 \sin^2(\theta(t;\lambda))\,dt.
\eeq
Because $\cot(\theta) = \frac{\psi_+'}{\psi_+}$, and due to the periodicity $\frac{\psi_+'(x+2d_+;\lambda)}{\psi_+(x+2d_+;\lambda)}=\frac{\psi_+'(x;\lambda)}{\psi_+(x;\lambda)}$ we have $\theta(2d_+;\lambda)=\theta(0;\lambda)+m\pi$, where due to continuity the value $m\in \Z$ is independent of $\lambda$.\footnote{In fact, it can be easily seen from Sturm oscillation theorem that $m=2n$, where $n$ is the index of the gap $G_n^+$.} Hence, $z(2d_+;\lambda)=z(0;\lambda)$. Using \eqref{E:rho_2d} and \eqref{E:z_var_const}, we thus obtain
\beq\label{E:z}
z(0)=z(2d_+) = e^{4d_+ \kappa}z(0) + \int_0^{2d_+}\left(\frac{\rho(t;\lambda)}{\rho(2d_+;\lambda)}\right)^2\sin^2(\theta(t;\lambda))\,dt.
\eeq
Because $\kappa>0$, we get $z(0)<0$ and conclude that $\theta(0;\lambda)$ is strictly decreasing throughout $G_n^+$. Therefore, $R_+(\lambda) = \cot(\theta(0;\lambda))$ is strictly increasing with respect to $\lambda$ throughout $G_n^+$.

In order to prove strict monotonicity of $R_-(\lambda)$, note that \eqref{E:rho_2d} is replaced by 
$\rho(-2d_-)=e^{-2d_-\kappa}\rho(0)$ and in \eqref{E:z} the value $2d_+$ is replaced by $-2d_-$ both in the arguments of the functions $z$ and $\rho$ and in the upper limit of the integral. This leads to the conclusion $z(0)>0$ which means that $R_-(\lambda)$ is strictly decreasing with respect to $\lambda$.
\epf

In order to apply the intermediate value theorem and prove crossing of the graphs of $R_+(\lambda)$ and $R_-(\lambda)$, we use their continuity within each gap and their limits as $\lambda$ approaches a gap edge.

\begin{lemma} \label{L:R_lim_vals}
Let $s\in \{s_1,s_2,\ldots\}$ be one of the boundary-points of the spectral gaps $G_n^\pm$ of $L_\pm$ respectively. If $s$ corresponds to a Dirichlet-eigenvalue of $L_\pm$ on $[0,d]$, then $\lim_{\lambda\to s, \lambda\in G_n^\pm}|R_\pm(\lambda)|=|R_\pm(s)|=\infty$ respectively, and if $s$ corresponds to a Neumann-eigenvalue of $L_\pm$ on $[0,d]$ then $\lim_{\lambda\to s, \lambda\in G_n^\pm} R_\pm(\lambda)=R_\pm(s)=0$ respectively.
\end{lemma}

\bpf 
We only consider the ``$+$'' case. Let $\lambda_k\in G_n^+$, $\lambda_k\to s$ be a given sequence. Due to \eqref{E:decaying_parts} the functions $\psi_+(\cdot,\lambda_k)$ have the form
$$
\psi_+(x;\lambda_k)=p_+(x;\lambda_k)e^{-\kappa_k x},
$$
where w.l.o.g. we may assume $\|p_+(\cdot;\lambda_k)\|_{L^\infty}=1$, which implies $\|\psi_+(\cdot;\lambda_k)\|_{L^\infty([0,\infty))}\leq 1$. On every compact subinterval $[0,b]\subset[0,\infty)$ the $H^2$-norm of $\psi_+(\cdot,\lambda_k)$ is uniformly bounded in $k$ and hence along a subsequence (again denoted by $\lambda_k$) the functions $\psi_+(\cdot;\lambda_k)$ converge in $H^1([0,b])$ (and hence, by the differential equation \eqref{E:L_system} also in $H^2([0,b])$) to a solution $v$ of $L_+ v = sv$ with $\|v\|_{L^\infty([0,b])}\leq 1$. Since this holds for every $b>0$, the function $v$ is a bounded solution of $L_+ v=sv$ on $[0,\infty)$ and therefore coincides with the bounded periodic Bloch function $p_1(x;s)$ in \eqref{E:Bloch_form_bdry}. The convergence of $R_+(\lambda_k)$ is now obvious by the embedding $H^2([0,b])$ into $C^1([0,b])$. 
\epf



To make the picture of the behavior of $R_\pm$ complete, it remains to determine their behavior at the lower end of the semi-infinite gap $G_0^\pm$, i.e. as $\lambda \rightarrow -\infty$.

\blem\label{L:R_asympt_infty}
Let $V_\pm$ be bounded potentials (not necessarily even, periodic or continuous). Then $R_\pm(\lambda)\rightarrow \mp \infty$ \ as \ $\lambda \rightarrow -\infty$.
\elem
\bpf \ The proof is, as for Lemma \ref{L:G_monotone}, shown only for $R_+$ with the one for $R_-$ being completely analogous. We rescale the Bloch function $\psi_+(x;\lambda)$ so that $\psi_+(0;\lambda)=1$. Note that this is possible if and only if $\psi_+(0;\lambda)\neq 0$, which we show to be true for all $\lambda \leq \inf V_+$. Suppose that $\psi_+(0;\lambda)= 0$. Testing $(L_+-\lambda)\psi_+=0$ with $\psi_+$ over $x\in [0,\infty)$, we get
$$\int_0^\infty (\psi_+')^2dx+\int_0^\infty(V_+-\lambda)\psi_+^2 dx = 0$$
and, therefore, $\lambda > \inf V_+$.

Let now $\lambda=-\nu^2$ for some $\nu >0$, s.t. $-\nu^2\in G_0^+$ and $-\nu^2\leq \inf V_+$, and define 
\beq \label{E:phi_def}
\phi_\nu(x):=\psi_+(x;-\nu^2)-e^{-\nu x}.
\eeq
We have 
\beq \label{E:phi_eq}
\phi_\nu''=\nu^2 \phi_\nu +V_+ \psi_+, \qquad \phi_\nu(0)=0.
\eeq

Since $R_+(\lambda)=R_+(-\nu^2)=\psi_+'(0;-\nu^2)=-\nu +\phi_\nu'(0)$, we need to determine the behavior of $\phi_\nu'(0)$ as $\nu \rightarrow \infty$. Using the Green's function, we solve \eqref{E:phi_eq} to obtain
$$\phi_\nu(x)=-\frac{1}{\nu}\left(e^{-\nu x}\int_0^x\sinh(\nu t)V_+(t)\psi_+(t;-\nu^2)dt +\sinh(\nu x)\int_x^\infty e^{-\nu t}V_+(t)\psi_+(t;-\nu^2)dt\right).$$
Therefore, $\phi_\nu'(0)  = -\int_0^\infty e^{-\nu t} V_+(t)\psi_+(t;-\nu^2)\,dt$ and
\beq \label{E:phi_prime_est}
|\phi_\nu'(0)|\leq \|V_+\|_{L^\infty} \|e^{-\nu \cdot }\|_{L^2(0,\infty)}\|\psi_+(\cdot;-\nu^2)\|_{L^2(0,\infty)} = \frac{\|V_+\|_{L^\infty}}{\sqrt{2\nu}}\|\psi_+(\cdot;-\nu^2)\|_{L^2(0,\infty)}.
\eeq
In order to estimate $\|\psi_+(\cdot;-\nu^2)\|_{L^2(0,\infty)}$, \eqref{E:phi_eq} yields
$$\nu^2 \|\phi_\nu\|^2_{L^2(0,\infty)} = -\|\phi_\nu'\|^2_{L^2(0,\infty)} -\int_0^\infty V_+(x)\psi_+(x;-\nu^2)\phi_\nu(x)dx,$$
implying $\nu^2 \|\phi_\nu\|^2_{L^2(0,\infty)} \leq \|V_+\|_{L^\infty} \|\psi_+(\cdot;-\nu^2)\|_{L^2(0,\infty)} \|\phi_\nu\|_{L^2(0,\infty)}$ and
\beq \label{E:phi_est}
\|\phi_\nu\|_{L^2(0,\infty)} \leq \frac{1}{\nu^2} \|V_+\|_{L^\infty} \|\psi_+(\cdot;-\nu^2)\|_{L^2(0,\infty)}.
\eeq
Therefore \eqref{E:phi_def} and \eqref{E:phi_est} together give
$\|\psi_+(\cdot;-\nu^2)\|_{L^2(0,\infty)} \leq   \frac{1}{\sqrt{2\nu}}+ \frac{1}{\nu^2} \|V_+\|_{L^\infty} \|\psi_+(\cdot;-\nu^2)\|_{L^2(0,\infty)}$. If $\nu^2>\|V_+\|_{L^\infty}$, we have the estimate
\beq \label{E:psi_plus_est}
\|\psi_+(\cdot;-\nu^2)\|_{L^2(0,\infty)} \leq \frac{(2\nu)^{-1/2}}{1-\nu^{-2} \|V_+\|_{L^\infty}}.
\eeq
Finally, combining \eqref{E:psi_plus_est} and \eqref{E:phi_prime_est}, we arrive at the bound
$$|\phi_\nu'(0)|\leq  \frac{(2\nu)^{-1}\|V_+\|_{L^\infty}}{1-\nu^{-2}\|V_+\|_{L^\infty}},$$
which implies $R_+(-\nu^2)=-\nu+\phi_\nu'(0) \rightarrow -\infty$ as $\nu \rightarrow \infty$.
\epf

The behavior of the ratio functions $R_\pm(\lambda)$ for the two examples  $V_+=V_-=\sin^2(\pi x/10)$ and $V_+=V_-=\cos^2(\pi x/10)$ is summarized in Figure \ref{F:R_behavior}. Note that Lemmas \ref{L:ordering}, \ref{L:G_monotone}, \ref{L:R_lim_vals} and \ref{L:R_asympt_infty} imply the behavior only for $\lambda \leq s_3$. The rest in Figure~\ref{F:R_behavior} is obtained without a rigorous proof from numerical computations of the gap edge eigenfunctions.
\begin{figure}[!ht]
\begin{center}
\includegraphics[scale=.6]{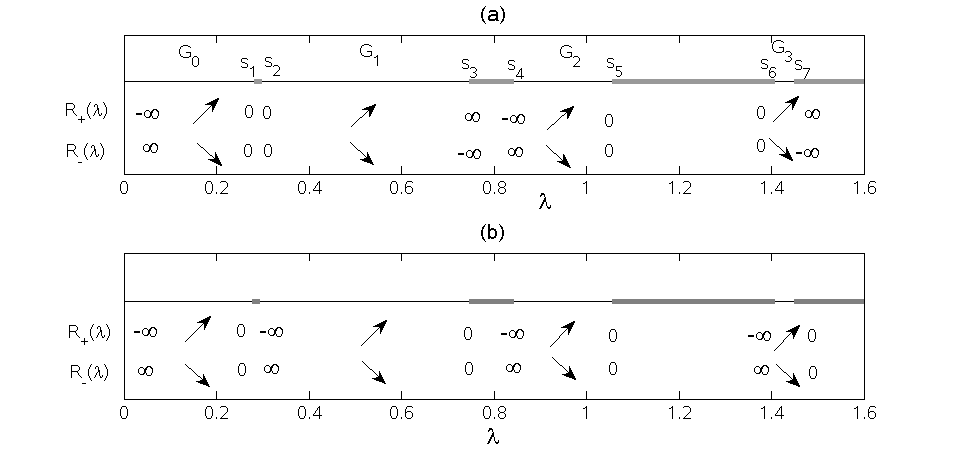}
\end{center}
\caption{Behavior of the ratio functions $R_\pm(\lambda)$ within gaps of $\sigma(L_\pm)$ for $V_+=V_-=\sin^2(\pi x/10)$ in (a) and for $V_+=V_-=\cos^2(\pi x/10)$ in (b). The arrows denote the monotonicity type.}
\label{F:R_behavior}
\end{figure}

By the intermediate value theorem and based on the behavior of $R_\pm$, we now obtain the following theorem, which has already been observed by Korotyaev \cite{KOR05}. 

\bthm\label{T:main} Let $G^-_n, G^+_m$ be two gaps in the spectrum of $L_-$ and $L_+$ respectively, such that $G^-_n \cap G^+_m \neq \emptyset$. Then the following two statements are equivalent:
\bi
\item[(a)] $\exists \ \lambda \in G^-_n \cap G^+_m$ \ such that $\lambda$ is an eigenvalue of $L$.
\item[(b)] Either $G^-_n=(\mu_n^-,\nu_{n+1}^-), G^+_m=(\nu_{m+1}^+,\mu_m^+)$ or $G^-_n=(\nu_{n+1}^-,\mu_n^-), G^+_m=(\mu_m^+,\nu_{m+1}^+)$.
\ei
In the affirmative case the eigenvalue is also unique.
\ethm

\noindent
{\em Remark.} Note that besides continuity and the limit values of $R_\pm(\lambda)$ their monotonicity is also needed to fulfill the conditions of the intermediate value theorem. Without monotonicity the ranges of the functions $R_+(\lambda)$ and $R_-(\lambda)$ on the intersection $G^-_n \cap G^+_{m}$ could be completely distinct, see Figure~\ref{F:R_intersec}~(a).
\begin{figure}[!ht]
\begin{center}
\includegraphics[scale=.5]{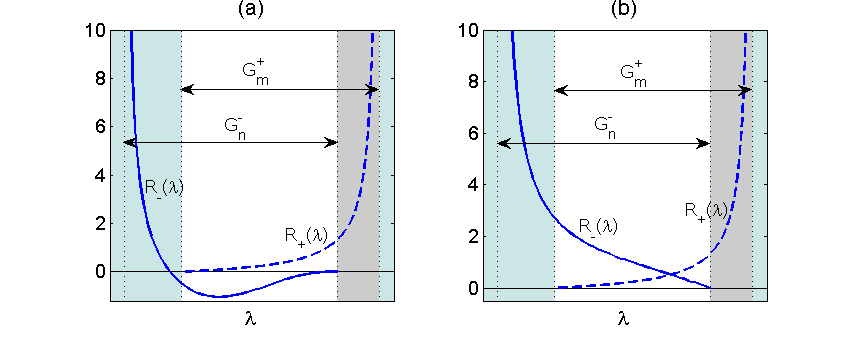}
\end{center}
\caption{A cartoon of the graphs of $R_-(\lambda)$ and $R_+(\lambda)$ when the former one of the two conditions in Theorem~\ref{T:main} holds. (a) No solution to \eqref{E:R_cond} without monotonicity of $R_-$ (hypothetical case). (b) Existence and uniqueness of the solution to \eqref{E:R_cond} on any $G^-_n \cap G^+_m$ with monotonicity of $R_\pm$.}
\label{F:R_intersec}
\end{figure}
With monotonicity of $R_\pm(\lambda)$ we, of course, obtain also uniqueness of solutions to~\eqref{E:R_cond}.

Let us call the gap $(\mu_n^\pm,\nu_{n+1}^\pm)$ a DN-gap and the gap $(\nu_{n+1}^\pm,\mu_n^\pm)$ an ND-gap. The semi-infinite gap belongs to the class of DN-gaps. The existence part of Theorem \ref{T:main} can then be formulated as follows:
\begin{quotation}
\noindent
\textit{
Whenever a DN/ND-gap of $L_-$ intersects an ND/DN-gap of $L_+$, respectively, a unique eigenvalue of $L$ exists in this intersection. 
}
\end{quotation}

\subsubsection{Example: additive interface}\label{S:example_add_jump}

The additive interface problem \eqref{E:L_prob} with \eqref{E:add_interf} is equivalent to \eqref{E:L_system} with $L_-=L_0:=-\partial_x^2+V_0(x)$ and $L_+=L_0+\alpha$.

Because $\sigma(L_0+\alpha)=\sigma(L_0)+\alpha$, we have $G_n^+=G_n^-+\alpha$, and because the Bloch functions of $L_0$ at the spectral parameter $\lambda$ are the same as the Bloch functions of $L_0+\alpha$ at $\lambda+\alpha$, to check the conditions of Theorem \ref{T:main}, one only needs to know $\sigma(L_0)$ and symmetries (even/odd) of the Bloch functions of $L_0$ at the gap edges $s_n$.

The existence part of Theorem \ref{T:main} can now be formulated as follows:
\begin{quotation}
\noindent
\textit{
Whenever $\alpha$ shifts the spectrum of $L_0$ so that a shifted DN/ND-gap intersects an (unshifted) ND/DN-gap, respectively, a unique eigenvalue of $L$ exists in this intersection. 
}
\end{quotation}

Theorem \ref{T:main} has several interesting and rather specific corollaries for the additive interface case. Firstly, clearly, if $|\alpha|<\alpha_*$, where $\alpha_*:=\inf_{n\in\N}(s_{2n}-s_{2n-1})$ stands for the width of the narrowest spectral band of $L_0$, the shift $\alpha$ is too small to make even the two gaps lying closest to each other overlap.
\bcor\label{C:unique_eval}
If $|\alpha|<\alpha_*:=\inf_{n\in\N}(s_{2n}-s_{2n-1})$, then $L$ has no eigenvalues.
\ecor

In the rest of this section $G_n$ denotes the $n$-th spectral gap of $L_0$.
As Lemma \ref{L:ordering} dictates, when $V_0$ is strictly increasing on $[0,d/2]$, the first finite gap $G_1=(s_2,s_3)$ is an ND-gap and thus if $\alpha$ shifts the semi-infinite (DN) gap $G_0$ so that $G_0+\alpha$ intersects $G_1$, an eigenvalue exists. Obviously, the infimal value of $\alpha>0$ achieving such an intersection is the width of the first spectral band $s_2-s_1$. Since $G_0$ is semi-infinite, there is no upper bound on $\alpha$ and if $\alpha>s_2-s_1$, the intersection is always nonempty. On  the other hand, when $V_0$ is decreasing on $[0,d/2]$, $G_1$ is a DN-gap and the intersection of $G_0+\alpha$ and $G_1$ contains no eigenvalues. As the next Corollary clarifies, for $\alpha<-(s_2-s_1)$ the situation is similar.
\bcor
Let $V_0$ be strictly increasing/strictly decreasing on $[0,\frac{d}{2}]$. If $|\alpha|> s_2-s_1$ then a unique eigenvalue/no eigenvalue of $L$ exists in $G_1\cap (G_0+\alpha)$ for $\alpha>0$ and in $G_0\cap (G_1+\alpha)$ for $\alpha<0$. 
\ecor

{\em Remark.}  For the case of the additive interface it is possible to show that \textit{the number of eigenvalues} of $L$ \textit{is finite} for any $\alpha \in \R$ based on the asymptotic behavior of gap locations and gap widths. Indeed, based on Theorem 4.2.2 in \cite{Eastham} the center of the $n$-th gap behaves like $cn^2 + o(n)$ as $n \rightarrow \infty$ with the constant $c\in \R$ dependent on $V_0$. The gap widths, on the other hand, tend to $0$ since they build an $l^2$ sequence, see Theorem 3 in \cite{GT84}. Therefore, asymptotically, the $n-$th gap has the form $c n^2 + J_n$, where both $\text{inf}(J_n)$ and $\text{sup}(J_n)$ behave like $o(n)$. For a given $\alpha \in \R$ infinitely many eigenvalues are thus possible only if 
for infinitely many pairs $(m,n) \in \N \times \N$ with $n\neq m$ there exist $s_n\in J_n$ and $t_m \in J_m$ such that 
\beq\label{E:inf_many_ev}
c n^2 + s_n = c m^2 +t_m + \alpha.
\eeq

As for $n=m$ no eigenvalues exist, we can rewrite \eqref{E:inf_many_ev} as
\[c(n+m)=\frac{\alpha+t_m-s_n}{n-m}.\]
Clearly, the right hand side is $o(n+m)$ while the left hand side is not. Thus only finitely many solutions of \eqref{E:inf_many_ev} exist.

For general interface problems (with $V(x) = \chi_{\{x<0\}}V_-(x)+\chi_{\{x\geq 0\}}V_+(x)$) the question of finiteness of the number of eigenvalues seems open. Due to Theorem 3 in \cite{GT84} there are, for example, potentials $V_-$ and $V_+$ with equal gap lengths and opposite DN/ND `polarities'. \textit{If}, in addition, the locations of the gap centers were identical, there would be an eigenvalue in each gap $G_n^- = G_n^+, \ n\in \N$. However, it seems to be an open problem whether such potentials $V_-$ and $V_+$ exist.

\paragraph{Numerical results}

The point spectrum of the additive interface problem with the potential $V_0(x) = \sin^2(\pi x/10)$ has been computed using a 4th order centered finite difference discretization. The eigenvalues are plotted in Figure~\ref{F:pt_spec} for a range of values of $\alpha$. The shaded regions are the union of spectral bands of $L_0$ and $L_0+\alpha$. The results agree with Theorem \ref{T:main}.
\begin{figure}[!ht]
\begin{center}
\includegraphics[scale=.49]{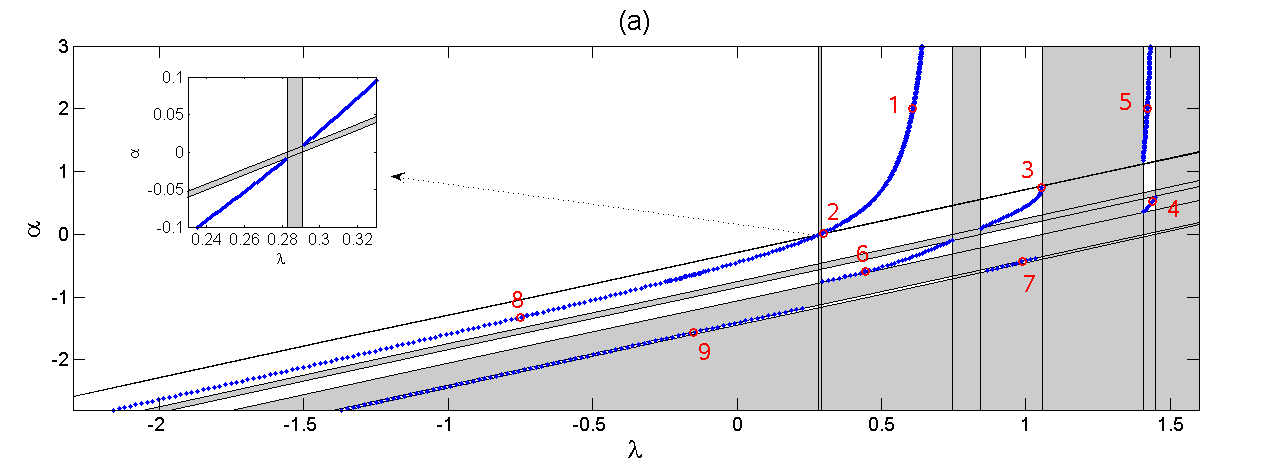}
\end{center}
\caption{Numerically computed point spectrum of $L$ with $V_0(x) = \sin^2(\pi x/10)$ for a range of values of $\alpha$. The union of spectral bands of $L_0$ and $L_0+\alpha$ is shaded. The inset blows up the region near $\lambda =s_1$, $\alpha=0$. Eigenfunctions for the labeled points are plotted in Figure~\ref{F:pt_spec_efns}.}
\label{F:pt_spec}
\end{figure}

In Figure~\ref{F:pt_spec_efns} we plot eigenfunctions corresponding to nine selected eigenvalues in Figure~\ref{F:pt_spec}. Note that the decay rate of the eigenfunctions is often very different on either side of the origin.
\begin{figure}[!ht]
\begin{center}
\includegraphics[scale=.54]{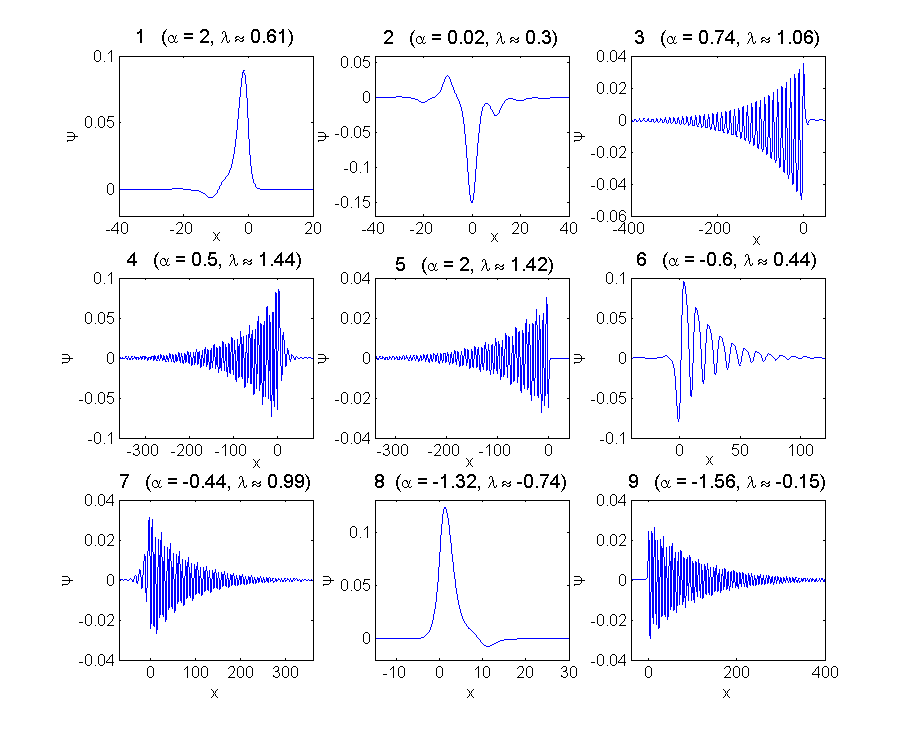}
\end{center}
\caption{Eigenfunctions corresponding to the 9 labeled eigenvalues in Figure~\ref{F:pt_spec}.}
\label{F:pt_spec_efns}
\end{figure}

For the potential $V_0(x)=\cos^2(\pi x /10)$ it is clear from the numerically obtained Figure~\ref{F:R_behavior} (b) that the intersections $G_j\cap (G_k +\alpha), \ j,k\in \{0, \ldots, 3\}$ contain no eigenvalues because the gaps $G_0$ and $G_1$ are DN-gaps and so seem to be $G_2$ and $G_3$. In other words, based on the numerics, the additive interface problem \eqref{E:L_prob}, \eqref{E:add_interf} with $V_0(x)=\cos^2(\pi x /10)$ has \textit{no eigenvalues on} $(-\infty, s_8]$. Note that our analysis guarantees non-existence of eigenvalues in $(-\infty,s_4]$.

\subsection{Point Spectrum for Interface Problems Made of Dislocated Even Potentials}\label{S:disloc_interf}
For the dislocation interface \eqref{E:transl_interf} we restrict our attention to the two representative cases $t=-s$ and $s=0$.

\subsubsection{Symmetric Dislocations}\label{S:disloc_sym}
Here we study the eigenvalue problem \eqref{E:L_prob} with \eqref{E:transl_interf} in the case where $t=-s, \ t\in (0,d)$. This can be done via the system
\beq\label{E:L_system2}
\begin{array}{ll}
L_-^t \psi^t := -\partial_x^2 \psi^t + V_0(x-t) \psi^t  = \lambda \psi^t & \text{for} \ x < 0, \\
L_+^t \psi^t := -\partial_x^2 \psi^t + V_0(x+t) \psi^t = \lambda \psi^t & \text{for} \ x \geq 0 \\
\end{array}
\eeq
coupled by the the $C^1$-matching conditions 
\beq\label{E:match2}
\psi^t(0-)=\psi^t(0+) \qquad \text{and} \qquad  \frac{d}{dx}\psi^t(0-) = \frac{d}{dx}\psi^t(0+).
\eeq
First note that the spectrum $\sigma(L_t)$ of the operator $L_t:= -\partial_x^2+V_0(x+t)$ on $\R$ is identical to the spectrum $\sigma(L_0)$ of $L_0=-\partial_x^2+V_0$ on $\R$ and we have $G_n^+=G_n^-$. Moreover, the Bloch functions $\psi_{1,2}^t$ of $L_t$ for $\lambda\in \R\setminus \sigma(L_t)=\R\setminus \sigma(L_0)$ are just shifts of the Bloch functions of $L_0$, i.e., $\psi_i^t(x;\lambda)=\psi_i^0(x+t;\lambda), i=1,2$. Therefore, an $L^2$-solution of \eqref{E:L_prob} with \eqref{E:transl_interf} can only exist if $\lambda \notin \sigma(L_0)$. For such $\lambda$ any localized eigenfunction $\psi^t$ of \eqref{E:L_prob} with \eqref{E:transl_interf} must take the form
\[\psi^t(x;\lambda)  = \chi_{\{x<0\}}\psi_-^t(x;\lambda) + \chi_{\{x\geq 0\}}\psi^t_+(x;\lambda), \]
where $\psi^t_\pm(x;\lambda)$ are those Bloch functions of $L_{\pm t}$, which decay on $\R^\pm$, respectively.

\medskip

As in Section \ref{S:jump_even_pot} we introduce the ratio functions
$$
R^t_\pm(x;\lambda) = \frac{\frac{\partial}{\partial x}\psi^t_\pm(x;\lambda)}{\psi^t_{\pm}(x;\lambda)},
$$
so that the matching conditions \eqref{E:match2} are equivalent to $R^t_+(0;\lambda)=R^t_-(0;\lambda)$. Due to the fact, that the Bloch functions $\psi^t_\pm$ are just shifts of the Bloch functions $\psi^0_\pm$, we see that $R^t_+(x;\lambda)=R^0_+(x+t;\lambda)$ and $R^t_-(x;\lambda)=R^0_-(x-t;\lambda)$. Thus, the matching condition \eqref{E:match2} amounts to 
$$
R^0_+(t;\lambda) = R^0_-(-t;\lambda).
$$
Finally, the evenness of the potential $V_0$ and the fact that only one linearly
independent Bloch function decaying at $+\infty$ exists, imply that
$\psi_+^0(x;\lambda) = \pm \psi_-^0(-x;\lambda)$ since $\lambda \notin \sigma(L_0)$, and
hence $R_+^0(t;\lambda) = -R_-^0(-t;\lambda)$ so that finding an eigenvalue of 
\eqref{E:L_prob} with \eqref{E:transl_interf} amounts to finding a zero or a pole of $R^0_+(t;\lambda)$ for some $t\in (0,d)$. This is done below via the intermediate value theorem and monotonicity properties of the function $R^0_+(t;\lambda)$. 

\medskip

For simplicity we write in the following $R(t;\lambda)$ instead of $R^0_+(t;\lambda)$. First, we need to generalize Lemma~\ref{L:G_monotone} on the monotonicity and continuity of $R(t;\lambda)$ or the corresponding Pr\"{u}fer angle $\theta(t;\lambda)$ as a function of $t$ and $\lambda$. Suppose $\psi \in L^2(0,\infty)$ solves $L_0\psi=\lambda\psi$. We apply again the Pr\"{u}fer transformation given by 
$$\psi(x;\lambda) = \rho(x;\lambda)\sin(\theta(x;\lambda)), \qquad \psi'(x;\lambda) = \rho(x;\lambda)\cos(\theta(x;\lambda)),$$ 
which transforms the equation $ L_0\psi=\lambda \psi$ into the system
\begin{align}
\theta' & = 1+(\lambda - V_0(x)-1)\sin^2(\theta), \label{pr1}\\
\rho' & = -\rho(\lambda - V_0(x)-1)\sin(\theta)\cos(\theta), \label{pr2}
\end{align}
where the prime denotes differentiation in $x$. Note that \eqref{E:Bloch_form} implies $2d$-periodicity in $t$ of $R(t;\lambda)$. Hence $\theta(t+2d;\lambda)=\theta(t;\lambda)+m\pi$, where $m$ is an integer which is constant in $\lambda$ within each spectral gap. In fact, it can be shown by the Sturm oscillation theorem that $m=2n$ when $\lambda\in G_n$. 

\medskip

In the subsequent arguments we use the following result on differential inequalities, cf. Walter~\cite{Wa}, which we quote in a slightly simplified way. Functions $v,w$ satisfying \eqref{sub_super} below are called sub-, supersolutions, respectively.

\blem 
\label{ivp_comparison}
Let $f:[a,b]\times\R\to \R$ be continuous and continuously differentiable with respect to the second variable. If $v,w\in C^1[a,b]$ satisfy 
\begin{equation}
v'\leq f(t,v), \quad w'\geq f(t,w) \mbox{ on } [a,b] \mbox{ with } v(a)\leq w(a),
\label{sub_super}
\end{equation}
then $v\leq w$ in $[a,b]$. More precisely, either $v<w$ in $(a,b]$ or there exists $c\in (a,b]$ such that $v=w$ on $[a,c]$ and $v<w$ on $(c,b]$. Moreover, if one of the differential inequalities holds strictly almost everywhere in $[a,b]$, then $v(t)<w(t)$ holds for all $t\in (a,b]$.
\elem

\blem[Monotonicity in $\lambda$]\label{L:G_monotone2}
Let $G_n=(s_{2n},s_{2n+1}), n\geq 0$  be a fixed gap. For $(t,\lambda)\in [0,d]\times \overline{G_n}$ the function $R$ is continuous except in the set $S=\bigcup_{t\in [0,d]} S_t$, where for each $t$ either $S_t=\emptyset$ or $S_t=\{(t,\lambda_t)\}$ and 
\begin{equation}
\label{blow_up}
\lim_{\lambda\to \lambda_t-} R(t;\lambda)=+\infty, \quad \lim_{\lambda\to \lambda_t+} R(t;\lambda)=-\infty.
\end{equation}
For a fixed $t$ the function $R(t;\lambda)$ is strictly increasing for $\lambda \in G_n$ if $S_t=\emptyset$, and strictly increasing for $\lambda \in (s_{2n},\lambda_t)$ and for $\lambda\in (\lambda_t,s_{2n+1})$ if $S_t=\{(t,\lambda_t)\}$. Moreover, if $\lambda,\mu \in \overline{G_n}$ and $S_t\neq \emptyset$, then $\lambda<\lambda_t<\mu$ implies $R(t;\lambda)>R(t;\mu)$. Consequently, for all $t\in [0,d]$ we have that $\lambda\neq\mu$ implies $R(t;\lambda)\neq R(t;\mu)$.
\elem

\bpf 
As we have seen in Lemma \ref{L:G_monotone}, the Pr\"ufer-variables $\theta$ and $\rho$ are continuous functions of both $t\in \R$ and $\lambda\in G_n$. Since $R(t;\lambda)=\cot(\theta(t;\lambda))$, the function $R(t;\lambda)$ is continuous except for those values, where $\theta(t;\lambda)$ passes through $k\pi, k \in \Z$. Since $R(t;\lambda)$ is strictly increasing in $\lambda$ at points of continuity by Lemma \ref{L:G_monotone}, the relation \eqref{blow_up} follows. The fact that there is at most one blow-up point $\lambda_{t_0}$ with respect to $\lambda$ will follow from the next statement. Let $\lambda_{t_0}$ be a pole and $\lambda<\lambda_{t_0}<\mu$ and suppose for contradiction that $R(t_0;\lambda)\leq R(t_0;\mu)$. By lowering $\mu$ if necessary and keeping the order $\lambda<\lambda_{t_0}<\mu$, we may achieve $R(t_0;\lambda)=R(t_0;\mu)$, i.e., there exists $k\in \Z$ such that $\theta(t_0;\lambda)=\theta(t_0;\mu)+k\pi$. Note that 
\begin{align*}
\theta'(t;\lambda) &= 1+(\lambda - V_0(t)-1)\sin^2(\theta(t;\lambda)),\\ 
\theta'(t;\mu) &= 1+(\mu - V_0(t)-1)\sin^2(\theta(t;\mu))> 1+(\lambda - V_0(t)-1)\sin^2(\theta(t;\mu)) 
\end{align*} 
for almost all $t\geq t_0$. By the comparison principle of Lemma \ref{ivp_comparison} we obtain $\theta(t;\lambda)<\theta(t;\mu)+k\pi$ for all $t>t_0$. Here we have used that $\theta$ and $\theta+k\pi$ solve the same differential equation. It follows in particular, that 
$$
\theta(t_0;\lambda)+m\pi=\theta(t_0+2d;\lambda)< \theta(t_0+2d;\mu)+k\pi=\theta(t_0;\mu)+(m+k)\pi
$$
contradictory to our assumption $\theta(t_0;\lambda)=\theta(t_0;\mu)+k\pi$. This proves the lemma.
\epf

\bcor 
\label{C:nr_evals}
For $t=-s$ the number of dislocation eigenvalues in any gap $G_n, n \geq 0,$ is $0$, $1$ or $2$. If there are $2$ eigenvalues, then one of them has an even and the other one an odd eigenfunction.
\ecor

\bpf
It follows from Lemma \ref{L:G_monotone2} that for fixed $t$ the function $R(t;\lambda)$ as a function of $\lambda$ can have at most one zero and at most one pole.
\epf

\blem[Monotonicity in $t$]\label{L:G_monotone3}
Suppose $V_0$ is an even, $d-$periodic $C^1$-function. Let $G_n=(s_{2n}, s_{2n+1}), n \geq 0,$ be a fixed gap and let $\lambda\in\partial G_n$. 
\begin{itemize} 
\item[(a)] If $V_0$ is strictly increasing on $[0,\frac{d}{2}]$, then either $\theta(t;\lambda)$ is strictly increasing for $t\in[0,d]$ or there exists $t_0\in (0,\frac{d}{2})$ such that $\theta(t;\lambda)$  is  strictly increasing for $t\in [0,t_0]\cup [d-t_0,d]$ and strictly decreasing for $t\in [t_0,d-t_0]$. 
\item[(b)] If $V_0$ is strictly decreasing on $[0,\frac{d}{2}]$, then either $\theta(t;\lambda)$ is strictly increasing for $t\in[0,d]$ or there exists $t_0\in (0,\frac{d}{2})$ such that $\theta(t;\lambda)$  is  strictly decreasing for $t\in [0,t_0]\cup [d-t_0,d]$ and strictly increasing for $t\in [t_0,d-t_0]$. 
\end{itemize}
Note that in both cases, $\theta(t;\lambda)$ can change monotonicity with respect to $t$ only once on $[0,d/2]$. \elem

\bpf
We give the proof in case (a). The proof for case (b) needs only minor modifications. Recall from Lemma~\ref{d_half} that for $\lambda\in\partial G_n$ the evenness of $V_0$ implies $\theta(\frac{d}{2};\lambda)=k\frac{\pi}{2}$ for some $k \in \Z$. Hence we have
$$ \theta\left(\frac{d}{2}+s;\lambda\right) = k\pi - \theta\left(\frac{d}{2}-s;\lambda\right)  \quad \forall s\in \left[0,\frac{d}{2}\right]
$$
since both sides satisfy the differential equation \eqref{pr1} with $V_0(x)=V_0(d/2+s)=V_0(d/2-s)$, and have the same initial values at $s=0$. In particular
\begin{equation}
\theta'\left(\frac{d}{2}+s;\lambda\right) = \theta'\left(\frac{d}{2}-s;\lambda\right)  \quad \forall s\in \left[0,\frac{d}{2}\right]
\label{symm}
\end{equation}
due to the evenness of $V_0(x)$ about $x=d/2$ (implied by $d-$periodicity and evenness about $x=0$).
In any of the two cases, the monotonicity of $\theta$ in $[0,\frac{d}{2}]$ has its counterpart in $[\frac{d}{2},d]$. Differentiation of \eqref{pr1} with respect to $t$ yields
\[
(\theta')' = 2(\lambda - V_0(t)-1)\sin(\theta)\cos(\theta) \theta'- V_0'(t)\sin^2(\theta) \label{d_theta}.
\]
If $\theta'(t_0)\geq 0$ for some $t_0\in(0,d/2]$, then by \eqref{symm} also $\theta'(d-t_0)\geq 0$, so that Lemma \ref{ivp_comparison} applied to $v:=0$ and $w:=\theta'$ on $[d-t_0,d]$ (note that $V_0'< 0$ a.e. on $[d/2,d]$) implies $\theta'>0$ on $(d-t_0,d]$ and by \eqref{symm} also on $[0,t_0)$. Below we show that such $t_0$ exists. Let $t_0$ be chosen maximal with these properties. If $t_0=d/2$, then $\theta$ is strictly increasing on $[0,d]$. If $t_0<d/2$, then $\theta'(t_0)=0$ and Lemma \ref{ivp_comparison} applied to $v:=\theta'$ and $w:=0$ on $[t_0,d/2]$ (note that $V_0'> 0$ a.e. on $[0,d/2]$) gives $\theta'<0$ on $(t_0,d/2]$ and by \eqref{symm} also on $[d/2,d-t_0)$. Consequently, $\theta'> 0$ on $[0,t_0)\cup (d-t_0,d]$, and $\theta'<0$ on $(t_0,d-t_0)$.

Note that the case $\theta'<0$ throughout $(0,d)$ is impossible since then $\theta(0)\neq \theta(d)$ and thus $\theta(s)$ is a multiple of $\pi$ for some $s\in[0,d]$ (because $\theta(0)$ and $\theta(d)$ are multiples of $\frac{\pi}{2}$ by the remarks before Lemma \ref{d_half}), whence $\theta'>0$ in some neighborhood of $s$ due to the differential equation for $\theta$, contradicting $\theta'<0$ on $(0,d)$. This proves the lemma.
\epf

\bthm\label{T:main_disloc_G0}
Suppose $V_0$ satisfies the basic assumptions, i.e., it is even, $d$-periodic and continuous. Let $s=-t$ in \eqref{E:transl_interf} and consider the semi-infinite gap $G_0=(-\infty, s_1)$.

\begin{itemize}
\item[(a)] If $V_0$ is strictly increasing on $[0,d/2]$, then there is no/exactly one dislocation eigenvalue in $G_0$ for $t\in [0,d/2]$ / $(d/2,d)$ respectively.
\item[(b)] If $V_0$ is strictly decreasing on $[0,d/2]$, then there is exactly one/no dislocation eigenvalue in $G_0$ for $t\in (0,d/2)$ / $[d/2,d]$ respectively.
\end{itemize}
\ethm
\bpf
It suffices to prove part (a), since (b) follows from (a) via shifting the potential by the half-period $\frac{d}{2}$ due to the evenness of $V_0(x)$ about $x=d/2$.
Recall that the first band edge $s_1$ is a Neumann eigenvalue. The first Neumann eigenfunction $u$ is positive, and hence, due to the $d-$periodicity and evenness of $V_0$ it has an extremum at $x=d/2$. It can be thus viewed as the first Neumann eigenfunction on the interval $x\in[0,d/2]$, i.e., the minimizer of the energy
\[\int_0^{d/2} {v'}^2+V_0(x) v^2\,dx, \quad \text{where} \  v \in H^1(0,d/2) \quad \text{with} \quad \int_0^{d/2}v^2 dx = 1.\]
As the decreasing rearrangement $u^*$ of $u$ decreases the energy, $u$ has to be decreasing, i.e., $u'(x)\leq 0$, on $[0,d/2]$. In fact, $u'(x) <0$ on $(0,d/2)$. If $u'(\xi)=0$ for some $\xi\in(0,d/2)$, then due to positivity of $u$ the function $R$ satisfies $R(0;s_1)=R(\xi;s_1)=R(d/2;s_1)=0$, hence $R$, and in turn $\theta$, change monotonicity at least three times on $(0,d/2)$, which is impossible by Lemma \ref{L:G_monotone3}. Therefore $R(t;s_1)<0$ for $t\in(0,d/2)$, $R(t;s_1)>0$ for $t\in(d/2,d)$, and $R(t;s_1)=0$ for $t\in \{0,d/2,d\}$.

Recall now from Lemma \ref{L:R_asympt_infty} that $R(t;\lambda)=R_+^t(0;\lambda) \rightarrow -\infty$ as $\lambda \rightarrow -\infty$ for any $t\in [0,d]$. Moreover, $R(t;\lambda)$ is continuous in $\lambda \in G_0$ because continuity can be broken only by a pole. But because 
$R(t;\lambda) \rightarrow -\infty$ as $\lambda \rightarrow -\infty$ and $R(t;\lambda)$ is increasing in $\lambda$ within each continuity segment, a pole would mean that $R(t;\lambda)$ takes the same value for some $\lambda_1\neq \lambda_2\in G_0$, which is impossible by Lemma \ref{L:G_monotone2}.

As a result $R(t;\lambda)$ stays negative for $t\in (0,d/2)$ throughout $\lambda\in G_0$, goes through $0$ once for $t\in (d/2,d)$, and takes the zero value at $\lambda =s_1 \notin G_0$ for $t=d/2$.
\epf

\begin{theorem}
\label{T:main_disloc_G1}
Suppose $V_0$ is an even, $d$-periodic $C^1$-function, let $s=-t$ in \eqref{E:transl_interf} and consider the first finite gap $G_1=(s_2, s_3)$.
\begin{itemize}
\item[(a)] Suppose $V_0$ is strictly increasing on $[0,\frac{d}{2}]$. Then $G_1=(\nu_2,\mu_1)$, and the second Neumann-eigenfunction is strictly monotone on $[0,d]$. For the first Dirichlet-eigenfunction $u$ we have the alternative: 
\begin{itemize}
\item[(a1)] $u$ is strictly monotone on $[0,\frac{d}{2}]$. Then there is exactly one dislocation-eigenvalue in $G_1$ for $t\in (0,d)\setminus\{\frac{d}{2}\}$ and none for $t=\frac{d}{2}$.
\item[(a2)] $u$ changes monotonicity on $[0,\frac{d}{2}]$ exactly once at the extremal point $d_0\in (0,\frac{d}{2})$. Then the number of dislocation-eigenvalues in $G_1$ is as follows: 
\begin{table}[!ht]
\begin{center}
\begin{tabular}{|l|c|c|c|c|c|}
\hline
dislocation parameter & $t\in (0,d_0)$ & $t\in [d_0,\frac{d}{2}]$, & $t\in (\frac{d}{2},d-d_0)$ 
& $t\in [d-d_0,d)$ & $t=d$\\
\hline
number of eigenvalues & 1 & 0 & 2 & 1 & 0\\
\hline
\end{tabular}
\end{center}
\end{table}
\end{itemize}
\item[(b)] Suppose $V_0$ is strictly decreasing on $[0,\frac{d}{2}]$. Then $G_1=(\mu_1,\nu_2)$, and the first Dirichlet-eigen\-function is strictly monotone on $[0,\frac{d}{2}]$. For the second Neumann-eigenfunction $u$ we have the alternative: 
\begin{itemize}
\item[(b1)] $u$ is strictly monotone on $[0,\frac{d}{2}]$. Then there is exactly one dislocation-eigenvalue in $G_1$ for $t\in (0,d)\setminus\{\frac{d}{2}\}$ and none for $t=\frac{d}{2}$.
\item[(b2)] $u$ changes monotonicity on $[0,\frac{d}{2}]$ exactly once at the extremal point $d_0\in (0,\frac{d}{2})$. Then the number of dislocation-eigenvalues in $G_1$ is as follows: 
\begin{table}[!ht]
\begin{center}
\begin{tabular}{|l|c|c|c|c|c|}
\hline
dislocation parameter & $t\in (0,d_0)$ & $t\in[d_0,\frac{d}{2})$, & $t=\frac{d}{2}$ & $t\in (\frac{d}{2},d-d_0)$ & $t\in [d-d_0,d]$\\
\hline
number of eigenvalues & 2 & 1 & 0 & 1 & 0\\
\hline
\end{tabular}
\end{center}
\end{table}
\end{itemize}
\end{itemize} 
\end{theorem}

\bpf
As in Theorem \ref{T:main_disloc_G0} it suffices to prove part (a) when, in addition, the roles of $\nu_2$ and $\mu_1$ are switched in the proof of (b). The strict monotonicity of the second Neumann eigenfunction and the fact that $G_1=(\nu_2,\mu_1)$ was already stated in Lemma~\ref{L:ordering}. For the monotonicity alternative of the first Dirichlet eigenfunction $u$ (which can be assumed positive on $(0,d)$) recall that 
$$u(x)=\rho(x;\mu_1)\sin\theta(x;\mu_1), \quad u'(x)=\rho(x;\mu_1)\cos\theta(x;\mu_1).
$$ 
We can assume that $\theta(0;\mu_1)=0$, $\theta(\frac{d}{2};\mu_1)=\frac{\pi}{2}$ and that $\theta(x;\mu_1)$ ranges in $[0,\pi)$ for $x\in [0,\frac{d}{2}]$. According to the monotonicity alternative for $\theta$ in Lemma~\ref{L:G_monotone3} there are two possibilities: either $\theta(x;\mu_1)$ is increasing and hence $\cos(\theta(x;\mu_1))>0$ for $x\in (0,\frac{d}{2})$, or $\theta(x;\mu_1)$ is strictly increasing for $t\in [0,t_0]$ and strictly decreasing for $t\in[t_0,\frac{d}{2}]$. In the latter case $\theta(x;\mu_1)$ crosses the value $\frac{\pi}{2}$ at some $d_0 \in (0,t_0)$ and hence $u'>0$ on $[0,d_0)$ and $u'<0$ on $(d_0,\frac{d}{2})$. This proves the monotonicity alternative (a1), (a2), and it remains to discuss the number of dislocation eigenvalues.

\smallskip

We may suppose that the second Neumann eigenfunction is strictly decreasing on $[0,d]$ with its unique zero at $\frac{d}{2}$. Thus $\theta(t;\nu_2)$ ranges within $[\frac{\pi}{2},\frac{3\pi}{2}]$ with $\theta(\frac{d}{2};\nu_2)=\pi$, $\theta'(\frac{d}{2};\nu_2)=1$. Therefore, $\theta$ increases near $d/2$ and taking into account Lemma~\ref{L:G_monotone3}(a), we find that $\theta$ must be strictly increasing on $[0,d]$ and hence $R(t;\nu_2)$ is strictly decreasing in $t$ on $[0,d/2)$ and on $(d/2,d]$ with 
\begin{equation}
\label{R_prop}
R(0+;\nu_2)=0-, \quad R\left(\frac{d}{2}-;\nu_2\right)=-\infty, \quad R\left(\frac{d}{2}+;\nu_2\right)=+\infty, \quad R(d-;\nu_2)=0+.
\end{equation}
\noindent
{\em Case (a1):} We may suppose that the first Dirichlet eigenfunction is strictly increasing and positive on $[0,\frac{d}{2}]$ and even around $\frac{d}{2}$. In this case $\theta(t;\mu_1)$ ranges through $[0,\frac{\pi}{2}]$ for $t\in [0,\frac{d}{2}]$ with $\theta(d/2,\mu_1)=\frac{\pi}{2}$ and through $[\frac{\pi}{2},\pi]$ for $t\in [\frac{d}{2},d]$. Hence $\theta'(d/2;\mu_1)\geq0$, and again Lemma~\ref{L:G_monotone3}(a) implies that $\theta$ is strictly increasing on $[0,d]$ and hence $R(t;\mu_1)$ is strictly decreasing in $t$ with 
$$
R(0+;\mu_1)=+\infty, \quad R\left(\frac{d}{2}-;\mu_1\right)=0+, \quad R\left(\frac{d}{2}+;\mu_1\right)=0-, \quad R(d-;\mu_1)=-\infty.
$$
For $t\in (0,\frac{d}{2})$ we have $R(t;\nu_2)<0<R(t;\mu_1)$ and hence by Lemma~\ref{L:G_monotone2} there is no pole $\lambda_t$ (i.e. $S_t=\emptyset$) and there exists a value $\lambda\in (\nu_2,\mu_1)$ with $R(t;\lambda)=0$, and this is the only zero. Thus we have the uniqueness of the dislocation eigenvalue. For $t=\frac{d}{2}$, the zero appears at $\lambda=\mu_1$ which is not inside the gap, i.e., there is no dislocation eigenvalue. Finally, for $t\in (\frac{d}{2},d)$ we have $R(t;\nu_2)>0>R(t;\mu_1)$ and hence by Lemma~\ref{L:G_monotone2} there exists a value $\lambda_t\in (\nu_2,\mu_1)$, where $R(t;\lambda)$ has a pole. No further poles or zeros can exists, which shows again uniqueness of the dislocation eigenvalue.

\smallskip

\noindent
{\em Case (a2):} We may suppose that the positive Dirichlet eigenfunction is strictly increasing on $[0,d_0]$, strictly decreasing on $[d_0,\frac{d}{2}]$ and even around $\frac{d}{2}$. In this case
$\theta(t;\mu_1)$ has the following properties:
$$
\begin{array}{llll}
\mbox{increasing from $0$ to $\frac{\pi}{2}$} & \mbox{for $t\in [0,d_0]$}, & \mbox{increasing from $\frac{\pi}{2}$ to $\theta^\ast$} & \mbox{for $t\in [d_0,t_0]$}, \\ 
\mbox{decreasing from $\theta^\ast$ to $\frac{\pi}{2}$} & \mbox{for $t\in [t_0,\frac{d}{2}]$}, & \mbox{decreasing from $\frac{\pi}{2}$ to $\theta_\ast$} & \mbox{for $t\in[\frac{d}{2},d-t_0]$},\\
\mbox{increasing from $\theta_\ast$ to $\frac{\pi}{2}$} & \mbox{for $t\in [d-t_0,d-d_0]$}, & \mbox{increasing from $\frac{\pi}{2}$ to $\pi$} & \mbox{for $t\in [d-d_0,d]$}
\end{array}
$$
for some $t_0\in(d_0,d/2)$, which translates into the following behavior of $R(t;\mu_1)$:
$$
\begin{array}{llll}
\mbox{decreasing from $+\infty$ to $0$} & \mbox{for $t\in [0,d_0]$}, & \mbox{decreasing from $0$ to $R_\ast$} & \mbox{for $t\in [d_0,t_0]$}, \\ 
\mbox{increasing from $R_\ast$ to $0$} & \mbox{for $t\in [t_0,\frac{d}{2}]$}, & \mbox{increasing from $0$ to $R^\ast$} & \mbox{for $t\in[\frac{d}{2},d-t_0]$},\\
\mbox{decreasing from $R^\ast$ to $0$} & \mbox{for $t\in [d-t_0,d-d_0]$}, & \mbox{decreasing from $0$ to $-\infty$} & \mbox{for $t\in [d-d_0,d]$}.
\end{array}
$$
If we combine this information with \eqref{R_prop}, we conclude:

(i) For $t\in (0,d_0)$ we have $R(t;\nu_2)<0<R(t;\mu_1)$ and hence by Lemma~\ref{L:G_monotone2} there exists a value $\lambda\in (\nu_2,\mu_1)$ with $R(t;\lambda)=0$. No other zero or pole can occur, which shows the uniqueness of the dislocation eigenvalue. 

(ii) For $t=d_0$ the zero has moved to the right-end of the gap, i.e., $R(d_0;\mu_1)=0$.

(iii) Next, we claim that 
\begin{equation}
R(t;\nu_2)<R(t;\mu_1)<0  \mbox{ for all } t\in (d_0,d/2)
\label{rel_R}
\end{equation}
This is obvious for $t$ near $d_0$ and has to hold by continuity for all $t\in (d_0,\frac{d}{2})$ since equality is excluded by Lemma~\ref{L:G_monotone2}. Moreover, \eqref{rel_R} also implies that there cannot be a pole of $R(t;\lambda)$ for $\lambda \in (\nu_2,\mu_1)$ by Lemma~\ref{L:G_monotone2}. Thus, $R(t;\lambda)$ increases continuously from $R(t;\nu_2)$ to $R(t;\mu_1)$ as $\lambda$ runs through $(\nu_2,\mu_1)$ with no zero or pole, i.e., there is no dislocation eigenvalue for $t\in [d_0,\frac{d}{2})$. 

(iv) For $t=\frac{d}{2}$ dislocation eigenvalues are excluded since $R(d/2;\nu_2+)=-\infty, R(d/2;\mu_1)=0$ and $R(d/2;\lambda)$ is strictly increasing for $\lambda \in [\nu_2,\mu_1]$. In fact, $t=\frac{d}{2}$ leads to a perfectly periodic $V(x)$ due to the symmetry of $V_0(x)$ about $x=\frac{d}{2}$. As a result, $t=\frac{d}{2}$ is no dislocation.

(v) Next we consider $t\in (\frac{d}{2},d-d_0)$. For such $t$ we claim that $R(t;\nu_2)>R(t;\mu_1)>0$. Whereas positivity is obvious, the ordering is clear for $t$ near $\frac{d}{2}$, cf. \eqref{R_prop}, and has to hold by continuity for all $t\in (\frac{d}{2},d-d_0)$ since equality is excluded by Lemma~\ref{L:G_monotone2}. Hence, there exists a pole $\lambda_{t} \in (\nu_2,\mu_1)$ of $R(t;\lambda)$ and also a zero $\lambda_0\in (\lambda_t,\mu_1)$, which yields exactly two dislocation eigenvalues for $t\in (\frac{d}{2},d-d_0)$. 

(vi) For $t=d-d_0$, the previous argument still shows the existence of a pole, but the zero has moved to the right end of the interval $(\nu_2,\mu_1)$ leaving us with only one dislocation eigenvalue. 

(vii) Next, consider $t\in (d-d_0,d)$. For such $t$ we have $R(t;\nu_2)>0>R(t;\mu_1)$ which forces the existence of a pole at some value $\lambda_t\in (\nu_2,\mu_1)$ with no further poles or zeros, i.e., there is exactly one dislocation eigenvalue. 

(viii) Finally, $t=d$ is the same as $t=0$ and corresponds to no dislocation and hence there are no eigenvalues. This completes the verification of the number of dislocation eigenvalues.
\epf

Finally, we give a partial answer to the question which of the cases (a1), (a2) or (b1), (b2) for a given potential $V_0$ actually occur. The condition given in the next theorem is a sufficient condition on the potential $V_0$ for (a2), (b2) to occur.

\bthm\label{T:Vbar_cond}
Suppose $V_0$ satisfies the basic assumptions, i.e., it is even, $d$-periodic, and continuous.
\begin{itemize}
\item[(i)] Assume that $V_0$ is strictly increasing on $[0,d/2]$  and
$$
V_0(x)\leq \overline{V}(x):= \beta+(\alpha-\beta)\left(\frac{2x}{d}-1\right)^2 \mbox{ for all } x\in [0,d/2],
$$
where $\beta:=V_0(\frac{d}{2})$ and $\alpha\in \R$ is arbitrary. If 
\begin{equation}\label{T:cond_2}
(\beta-\alpha)d^2 > 80(13-2\sqrt{37})\approx 66.75,
\end{equation}
then only the case (a2) of Theorem \ref{T:main_disloc_G1} occurs, i.e., the first Dirichlet-eigenfunction on $[0,d]$ is even around $\frac{d}{2}$ but changes its monotonicity at some $d_0\in (0,\frac{d}{2})$.

\item[(ii)] Assume that $V_0$ is strictly decreasing on $[0,d/2]$  and
$$
V_0(x)\leq \overline{V}(x):= \beta+(\alpha-\beta)\frac{4}{d^2}x^2 \mbox{ for all } x\in [0,d/2],
$$
where $\beta:=V_0(0)$ and $\alpha\in\R$ is arbitrary. If 
\begin{equation}\label{T:cond_3}
(\beta-\alpha)d^2 > 80(13-2\sqrt{37})\approx 66.75,
\end{equation}
then only the case (b2) of Theorem \ref{T:main_disloc_G1} occurs, i.e., the second Neumann-eigenfunction on $[0,d]$ is odd around $\frac{d}{2}$ but changes its monotonicity at some $d_0\in (0,\frac{d}{2})$.
\end{itemize}
\ethm

\noindent
{\em Remark.} It will become clear from the proof that \eqref{T:cond_2} and \eqref{T:cond_3} are not the only conditions that lead to the conclusion of the theorem. In fact, by choosing different upper bounds $\overline{V}$ and a different candidate function $w(x)$ in the proof below, one may obtain sufficient conditions which are different from \eqref{T:cond_2} and \eqref{T:cond_3}. Since there are manifold ways to derive such conditions, we decided to give only the simplest one. Nevertheless, \eqref{T:cond_2} and \eqref{T:cond_3} are already sufficient to cover example potentials such as $V_0(x)=\sin^2(\pi x/10)$ and $V_0(x)=\cos^2(\pi x/10)$, respectively.

\bpf
Suppose $V_0$ is increasing on $[0,\frac{d}{2}]$. Let $\mu_1$ be the first Dirichlet eigenvalue on $[0,d]$ with corresponding positive eigenfunction $u$. Then $\mu_1=\kappa_{DN}$, where $\kappa_{DN}$ denotes the first eigenvalue on $[0,\frac{d}{2}]$ with Dirichlet boundary condition at $0$ and Neumann boundary condition at $\frac{d}{2}$. Let $\theta$ be the Pr\"ufer-angle for $u$ normalized by $\theta(0)=0$, which implies $\theta(\frac{d}{2})=\frac{\pi}{2}$. The non-monotonicity of $u$ can be shown by proving that $\theta'(\frac{d}{2})<0$, i.e., $\theta>\frac{\pi}{2}$ in a left-neighborhood of $\frac{d}{2}$. By the differential equation for $\theta$ we obtain 
$$
\theta'(\tfrac{d}{2})=1+\left(\mu_1-V_0\left(\tfrac{d}{2}\right)-1\right)\sin^2 \theta\left(\tfrac{d}{2}\right)=\mu_1-V_0\left(\tfrac{d}{2}\right).
$$
Using the variational characterization of $\mu_1=\kappa_{DN}$, it suffices to find one function $w\in H^1(0,\frac{d}{2})$ with $w(0)=0$ such that 
\begin{equation}
\label{cond_ev}
\frac{\int_0^\frac{d}{2} {w'}^2+V_0(x) w^2\,dx}{\int_0^\frac{d}{2} w^2\,dx}< V_0\left(\tfrac{d}{2}\right)=\beta.
\end{equation}
Using the upper bound $V_0(x)\leq \overline{V}(x)$ and the quadratic candidate function $w(x)=x(2c-x)$ with $c\in\R$ to be determined, condition \eqref{cond_ev} amounts to
\begin{multline}
\int_0^\frac{d}{2} {w'}^2+\overline{V}(x) w^2\,dx-\beta\int_0^\frac{d}{2} w^2\,dx \\= 
\frac{d}{3360}\bigl(56(120-\gamma)c^2-14(240-\gamma)c+(560-\gamma)d^2\bigr)<0, \quad \mbox{ where } \gamma := (\beta-\alpha)d^2.
\label{cond_c}
\end{multline}
If $\gamma\geq 120$, then \eqref{cond_c} can always be achieved by an appropriate choice of $c$. If $\gamma<120$, then the optimal choice for $c$ is $c=\frac{d(\gamma-240)}{8(\gamma-120)}$ and hence \eqref{cond_c} amounts to 
$$
\frac{d^3(\gamma^2-2080\gamma+134400)}{8\cdot3360(120-\gamma)}<0,
$$
which is fulfilled for $\gamma\in\bigl(80(13-2\sqrt{37}),120\bigr)\approx(66.75,120)$. Altogether, the statement (i) of the theorem holds true for $\gamma=(\beta-\alpha)d^2>80(13-2\sqrt{37})$. This concludes the proof of statement (i). Part (ii) can be obtained from part (i) by reflecting the interval $[0,\frac{d}{2}]$. 
\epf

\paragraph{Numerical Results}\label{S:numerics_sym_disloc}

We present results of numerical computations of the point spectrum of $L$ with the dislocation interface \eqref{E:transl_interf} with $s=-t$ and $V_0=\sin^2(\pi x/10)$ as well as $V_0=\cos^2(\pi x/10)$. 

As one can see in Figure \ref{F:pt_spec_disloc} bottom, the number of eigenvalues in the semi-infinite gap agrees with Theorem \ref{T:main_disloc_G0}. Regarding eigenvalues in the first finite gap $G_1=(s_2,s_3)$, note first that since $V_0=\sin^2(\pi x/10)$ satisfies the conditions of Theorem \ref{T:Vbar_cond} (with $\beta=1$ and, for instance, $\alpha=0.3$), we know that the first Dirichlet eigenfunction changes monotonicity at some $d_0\in (0,d/2)=(0,5)$. The case $(a2)$ of Theorem \ref{T:main_disloc_G1}, therefore, applies. We obtain numerically $d_0\approx 2.16$, see Fig. \ref{F:pt_spec_disloc} top. The number of eigenvalues in the gap $G_1$ agrees with the theory at each $t\in (0,d)$, see Figure~\ref{F:pt_spec_disloc} bottom. Eigenvalues in the gaps $G_2$ and $G_3$ are also plotted; note that for these our analysis provides no explanation other than the statement of Corollary \ref{C:nr_evals}.

Figure~\ref{F:pt_spec_disloc_efns} shows the eigenfunctions corresponding to the 9 labeled eigenvalues in Figure~\ref{F:pt_spec_disloc}.

\begin{figure}[!ht]
\begin{center}
\includegraphics[scale=.47]{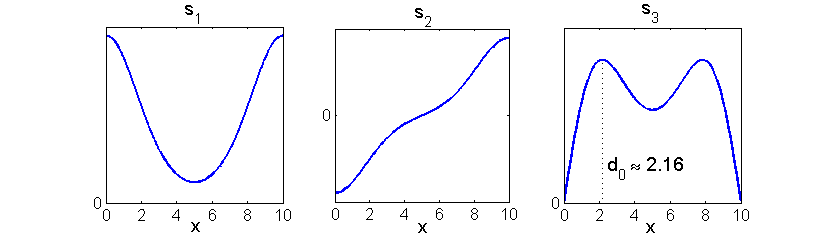}\\
\includegraphics[scale=.44]{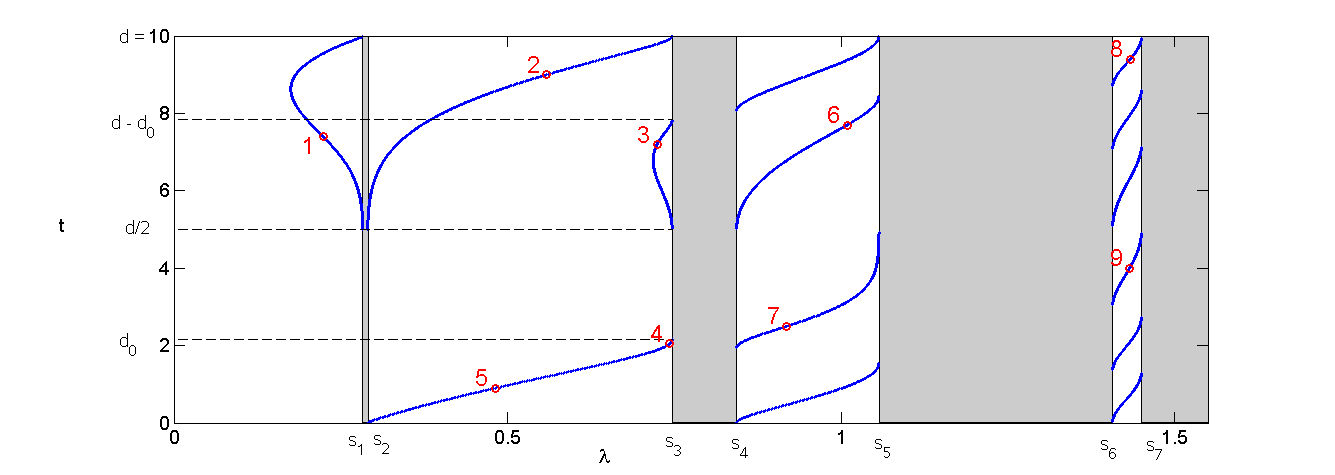}
\end{center}
\caption{top: the first three band edge Bloch functions of $L_0$ with $V_0(x) = \sin^2(\pi x/10)$;  bottom: point spectrum of $L$ for \eqref{E:transl_interf} with $s=-t$ and $V_0=\sin^2(\pi x/10)$ for $t\in [0,d)$. The spectral bands of $L$ are shaded. Eigenfunctions for the labeled points are plotted in Figure~\ref{F:pt_spec_disloc_efns}.}
\label{F:pt_spec_disloc}
\end{figure}

\begin{figure}[!ht]
\begin{center}
\includegraphics[scale=.48]{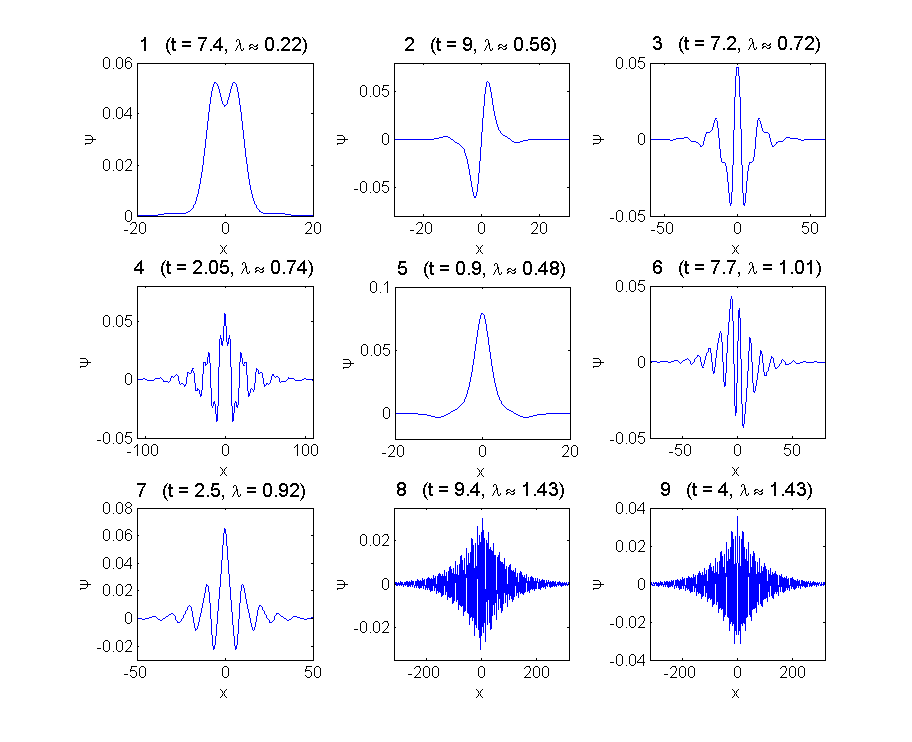}
\end{center}
\caption{Eigenfunctions corresponding to the 9 labeled eigenvalues in Figure~\ref{F:pt_spec_disloc}.}
\label{F:pt_spec_disloc_efns}
\end{figure}

The results for $V_0(x) = \cos^2(\pi x/10)$, as an example of a potential that falls in the case (b) of Theorem \ref{T:main_disloc_G1}, are, in fact, contained in the lower part of Figure~\ref{F:pt_spec_disloc} because $\cos^2(\pi (x-t)/10) =  \sin^2(\pi (x-(t+5))/10)$. As $\cos^2(\pi x/10)$ satisfies the conditions of Theorem \ref{T:cond_2} (with $\beta=1$ and, for instance, $\alpha=0.3$), we know that the alternative (b2) has to apply.

\subsubsection{One-sided Dislocations}\label{S:disloc_one_side}
 As the second representative example of the dislocation problem \eqref{E:L_prob} with \eqref{E:transl_interf} we choose $s=0, t\in (0,d)$, which is equivalent to the system
\beq\label{E:L_system3}
\begin{array}{ll}
L_-^0 \psi^t := -\partial_x^2 \psi^t + V_0(x) \psi^t  = \lambda \psi^t & \text{for} \ x < 0, \\
L_+^t \psi^t := -\partial_x^2 \psi^t + V_0(x+t) \psi^t = \lambda \psi^t & \text{for} \ x \geq 0 \\
\end{array}
\eeq
coupled by the the $C^1$-matching conditions \eqref{E:match2}.
Localized eigenfunctions $\psi^t$, once again, exist only for $\lambda \notin \sigma(L_0)$ and have the form
\[\psi^t(x;\lambda)  = \chi_{\{x<0\}}\psi_-^0(x;\lambda) + \chi_{\{x\geq 0\}}\psi^t_+(x;\lambda), \]
where $\psi^0_-(x;\lambda)$ and $\psi^t_+(x;\lambda)$ are those Bloch functions of $L_0$ and $L_t=-\partial_x^2+V_0(x+t)$, which decay on $\R^-$ and $\R^+$, respectively. The matching condition  \eqref{E:match2} now becomes
$$
R^0_+(t;\lambda) = R^0_-(0;\lambda),
$$
where $R^0_-(0;\lambda)$ is the same as $R_-(\lambda)$ defined in \eqref{E:R_cond}.

Because $R^0_-(0;\lambda)$ is decreasing and continuous in each gap (Lemma \ref{L:G_monotone}) and given the analysis of $R^0_+(t;\lambda)$ in Section \ref{S:disloc_sym}, determining intersections of $R^0_+(t;\lambda)$ and $R^0_-(0;\lambda)$ in $G_0$ and $G_1$ is now straightforward.

\blem
For $s=0$ the number of dislocation eigenvalues in any gap $G_n, n \geq 0,$ is $0,1$ or $2$. 
\elem
\bpf
$R_+^0(t;\lambda)$ is strictly increasing and continuous in $\lambda$ on each continuity segment and its continuity can be broken only at one point (pole) in $G_n$, see Lemma \ref{L:G_monotone2}. As $R^0_-(0;\lambda)$ is continuous and decreasing throughout $G_n$, only up to 2 intersections of $R_+^0(t;\lambda)$ and $R^0_-(0;\lambda)$ can occur.
\epf

\bthm\label{T:G0_one_sided}
Suppose $V_0$ satisfies the basic assumptions, i.e., it is continuous, even and $d$-periodic, and let $s=0$ in \eqref{E:transl_interf}, and consider the semi-infinite gap $G_0=(-\infty, s_1)$.

\begin{itemize}
\item[(a)] If $V_0$ is strictly increasing on $[0,d/2]$, then there is no/exactly one dislocation eigenvalue in $G_0$ for $t\in [0,d/2]$ / $(d/2,d)$ respectively.
\item[(b)] If $V_0$ is strictly decreasing on $[0,d/2]$, then there is exactly one/no dislocation eigenvalue in $G_0$ for $t\in (0,d/2)$ / $[d/2,d]$ respectively.
\end{itemize}
\ethm

\bpf
We, once again, present the proof only of (a) as (b) follows by shifting the potential in $x$ (or $t$) by $d/2$. As explained in the proof of Theorem \ref{T:main_disloc_G0}, 
$s_1$ is a Neumann eigenvalue and the corresponding eigenfunction can be taken positive on $[0,d/2]$ with $u'<0$ on $(0,d/2)$ and with a point of even symmetry at $x=d/2$. 

By Lemmas \ref{L:G_monotone}, \ref{L:R_lim_vals} and \ref{L:R_asympt_infty} the function $R^0_-(0;\lambda)$ decreases continuously from $\infty$ at $\lambda \rightarrow -\infty$ to $0$ at $\lambda=s_1$. The behavior of $R^0_+(t;\lambda)$ is explained in the proof of Theorem \ref{T:main_disloc_G0}. It follows that $R^0_-(0;\lambda)$ and $R^0_+(t;\lambda)$ intersect in $G_0$ exactly once for $t\in (d/2,d)$ and do not intersect for  $t\in [0,d/2]$.
\epf

\begin{theorem}
\label{T:G1_one_sided}
Suppose $V_0$ is an even, $d$-periodic $C^1$-function, let $s=0$ in \eqref{E:transl_interf}, and consider the first finite gap $G_1=(s_2, s_3)$.
\begin{itemize}
\item[(a)] If $V_0$ is strictly increasing on $[0,\frac{d}{2}]$, and hence $G_1=(\nu_2,\mu_1)$, then there is exactly one dislocation-eigenvalue in $G_1$ for all $t\in (0,d)$.
\item[(b)] If $V_0$ is strictly decreasing on $[0,\frac{d}{2}]$, and hence $G_1=(\mu_1,\nu_2)$, then we have the following alternative for the second Neumann-eigenfunction $u$: 
\begin{itemize}
\item[(b1)] $u$ is strictly monotone on $[0,\frac{d}{2}]$. Then there is exactly one dislocation-eigenvalue in $G_1$ for all $t\in (0,d)$.
\item[(b2)] $u$ changes monotonicity on $[0,\frac{d}{2}]$ exactly once at the extremal point $d_0\in (0,\frac{d}{2})$. Then the number of dislocation-eigenvalues in $G_1$ is as follows: 
\begin{table}[!h]
\begin{center}
\begin{tabular}{|l|c|c|c|c|c|}
\hline
dislocation parameter & $t\in (0,d_0)$ & $t\in[d_0,d-d_0)$ & $t\in [d-d_0,d)$\\
\hline
number of eigenvalues & 2 & 1 & 0\\
\hline
\end{tabular}
\end{center}
\end{table}
\end{itemize}
\end{itemize} 
\end{theorem}

\bpf
{\em Case (a):} As explained in the proof of part (a) of Theorem \ref{T:main_disloc_G1}, for $t\in(0,d/2]$ we have $R_+^0(t;\nu_2+)<0$ and $R_+^0(t;\lambda)$ continuous and increasing in $\lambda \in G_1$. Therefore, $R_+^0(t;\lambda)$ intersects $R_-^0(0;\lambda)$ exactly once, as $R_-^0(0;\lambda)$ decreases continuously from $0$ at $\lambda=\nu_2+$ to $-\infty$ at $\lambda= \mu_1-$, see Lemmas \ref{L:G_monotone}, \ref{L:R_lim_vals}.

Next, as the proof of  Theorem \ref{T:main_disloc_G1} (a) shows, for $t\in(d/2,d)$ the function $R_+^0(t;\lambda)$ has a pole at some $\lambda_t\in G_1$ and increases continuously on the interval
$(\nu_2,\lambda_t)$ with $R_+^0(t;\nu_2)>0, R_+^0(t;\lambda_t-)=\infty$ and on the interval $(\lambda_t,\mu_1)$ with $R_+^0(t;\lambda_t+)=-\infty$. The functions $R_+^0(t;\lambda)$ and $R_-^0(0;\lambda)$, therefore, intersect exactly once on $\lambda\in(\lambda_t,\mu_1)$ and they do not intersect on $\lambda\in(\nu_2,\lambda_t)$.

\noindent
{\em Case (b):} In the case of $V_0$ strictly decreasing on $[0,d/2]$ the function $R_-^0(0;\lambda)$ is continuous and strictly decreasing from $\infty$ at $\lambda = \nu_2+$ to $0$ at $\lambda=\mu_1$, see Lemmas \ref{L:G_monotone}, \ref{L:R_lim_vals}. We obtain below the behavior of $R_+^0(t;\lambda)$ from that of $R(t;\lambda)$ in the proof of Theorem \ref{T:main_disloc_G1} (a) by the shift of $d/2$ in $t$ and switching of the roles of $\mu_1$ and $\nu_2$. 
 
{\em Case (b1):} For $t\in (0,d/2)$ we have $R_+^0(t;\mu_1)>0>R_+^0(t;\nu_2)$ and $R_+^0(t;\lambda)$ has one pole in $\lambda$ within $G_1$. $R_-^0(0;\lambda)$ thus intersects $R_+^0(t;\lambda)$ exactly once on $G_1$. For $t\in [d/2,d)$ the function $R_+^0(t;\lambda)$ is continuous on $G_1$ and $R_+^0(t;\mu_1)\leq 0<R_+^0(t;\nu_2-)$. Exactly one intersection of $R_+^0(t;\lambda)$ and $R_-^0(0;\lambda)$ thus exists.

{\em Case (b2):} For  $t\in (0,d_0)$ the function $R_+^0(t;\lambda)$ behaves  in $\lambda$ like $R(t;\lambda)$ on $t\in (d/2,d-d_0)$ in (v) in the proof of Theorem~\ref{T:main_disloc_G1} (a2). Note that $d_0$ here corresponds to $d/2-d_0$ in the proof of Theorem~\ref{T:main_disloc_G1}. Namely, we get $R_+^0(t;\mu_1)>R_+^0(t;\nu_2)>0$ and a pole of $R_+^0(t;\lambda)$ at some $\lambda_t\in G_1$. Two intersections of $R_+^0(t;\lambda)$ and $R_-^0(0;\lambda)$ thus exist. For $t\in [d_0,d-d_0)$ the behavior of the eigenfunction and hence of $R_+^0(t;\lambda)$ is qualitatively the same as in (b1) of this proof and precisely one eigenvalue thus appears in $G_1$. Finally, for $t\in [d-d_0,d)$ the function $R_+^0(t;\lambda)$ behaves in $\lambda$ like $R(t;\lambda)$ on $t\in [d_0,d/2)$ in (iii) in the proof of Theorem~\ref{T:main_disloc_G1} (a2). Therefore, $R_+^0(t;\mu_1)<R_+^0(t;\nu_2)\leq 0$ and $R_+^0(t;\lambda)$ is continuous throughout $G_1$. No intersections of $R_+^0(t;\lambda)$ and $R_-^0(0;\lambda)$ thus occur. Finally, $t=d$ corresponds to no dislocation resulting in a purely continuous spectrum of $L$.
\epf

\paragraph{Numerical Results}\label{S:numerics_one_side_disloc}

Results of numerical eigenvalue computations with the dislocation interface \eqref{E:transl_interf} with $s=0$ and $V_0=\sin^2(\pi x/10)$ are displayed in Figure~\ref{F:pt_spec_disloc_one_side}. They agree with Theorems \ref{T:G0_one_sided} and \ref{T:G1_one_sided}. Figure~\ref{F:efns_disl_one_side} shows the eigenfunctions corresponding to the 6 labeled eigenvalues in Figure~\ref{F:pt_spec_disloc_one_side}. As expected, they lack symmetry in contrast with the eigenfunctions of the symmetric dislocation in Figure~\ref{F:pt_spec_disloc_efns}.

\begin{figure}[!ht]
\begin{center}
\includegraphics[scale=.4]{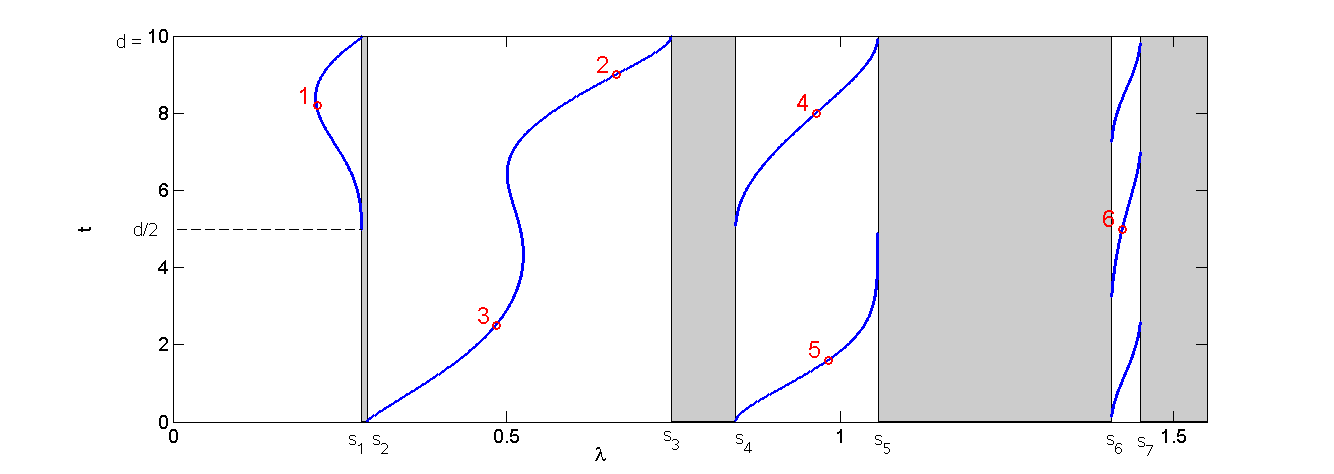}
\end{center}
\caption{Point spectrum of $L$ for \eqref{E:transl_interf} with $s=0$ and $V_0=\sin^2(\pi x/10)$. Eigenfunctions for the labeled points are plotted in Figure~\ref{F:efns_disl_one_side}.}
\label{F:pt_spec_disloc_one_side}
\end{figure}

\begin{figure}[!ht]
\begin{center}
\includegraphics[scale=.5]{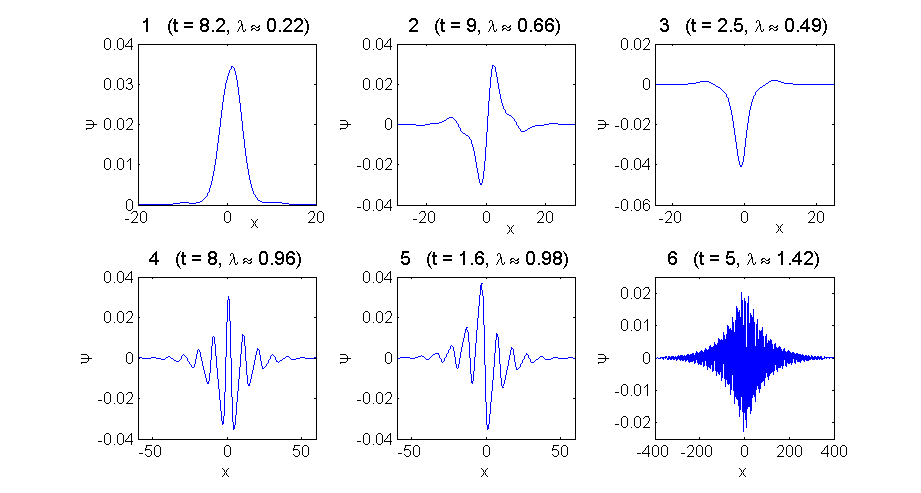}
\end{center}
\caption{Eigenfunctions corresponding to the 6 labeled eigenvalues in Figure~\ref{F:pt_spec_disloc_one_side}.}
\label{F:efns_disl_one_side}
\end{figure}

The results for $V_0(x) = \cos^2(\pi x/10)$, as an example of a potential that falls in the case (b) of Theorem \ref{T:G1_one_sided}, appear in Figures \ref{F:pt_spec_disloc_one_side_cos} and \ref{F:efns_disl_one_side_cos}. As we know from Section \ref{S:disloc_sym}, the potential $\cos^2(\pi x/10)$ falls into the case (b2) and the second Neumann eigenfunction thus changes monotonicity on $(0,d/2)$, see Figure~\ref{F:pt_spec_disloc_one_side_cos} top. Agreement of the numerics with Theorems \ref{T:G0_one_sided} and \ref{T:G1_one_sided} is, once again, observed.
\begin{figure}[!ht]
\begin{center}
\includegraphics[scale=.45]{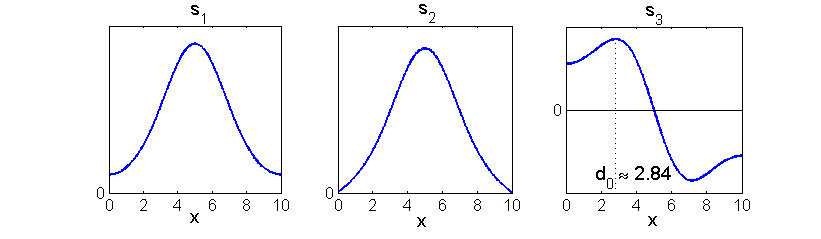}\\
\includegraphics[scale=.4]{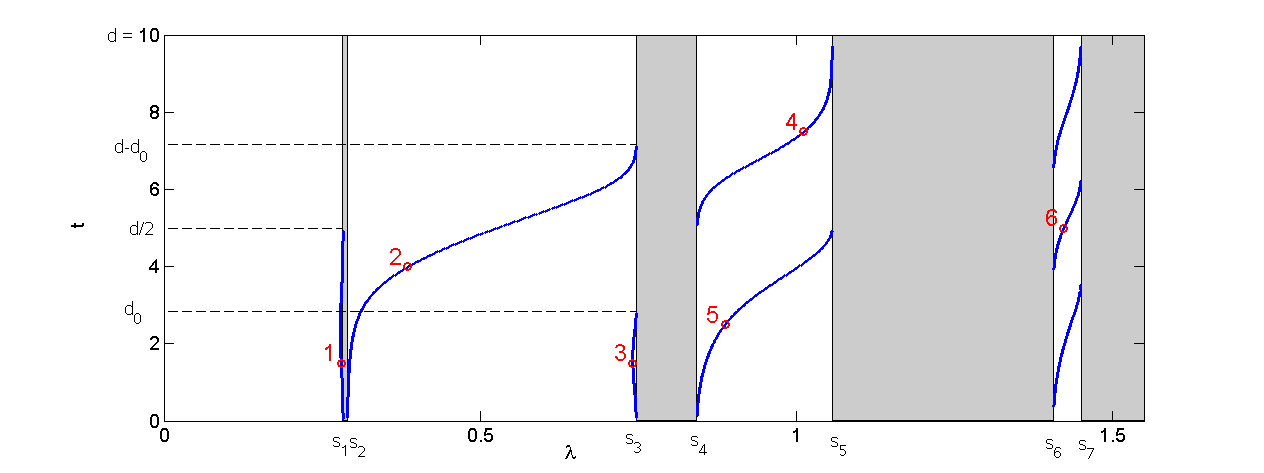}
\end{center}
\caption{top: the first three band edge Bloch functions of $L_0$ with $V_0(x) = \cos^2(\pi x/10)$;  bottom: point spectrum of $L$ for \eqref{E:transl_interf} with $s=0$ and $V_0=\cos^2(\pi x/10)$ for $t\in [0,d)$. The spectral bands of $L$ are shaded. Eigenfunctions for the labeled points are plotted in Figure~\ref{F:efns_disl_one_side_cos}.}
\label{F:pt_spec_disloc_one_side_cos}
\end{figure}

\begin{figure}[!ht]
\begin{center}
\includegraphics[scale=.5]{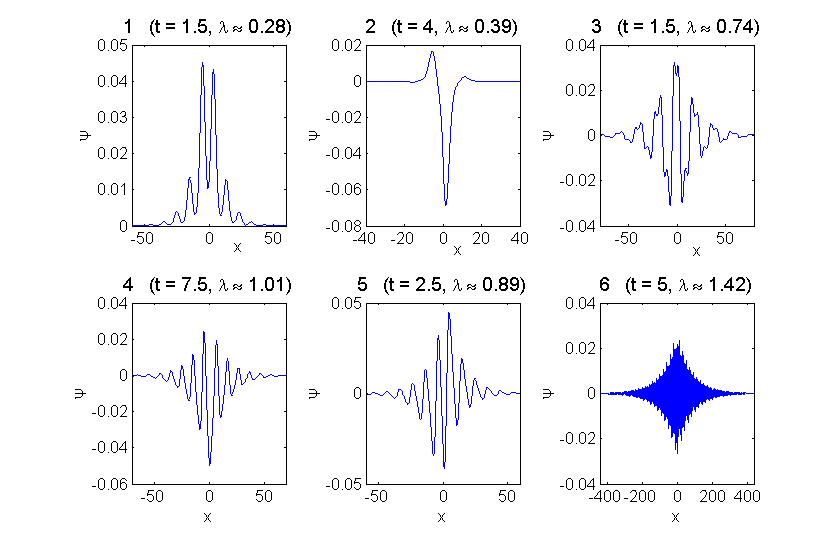}
\end{center}
\caption{Eigenfunctions corresponding to the 6 labeled eigenvalues in Figure~\ref{F:pt_spec_disloc_one_side_cos}.}
\label{F:efns_disl_one_side_cos}
\end{figure}



{\bf Acknowledgements} 
The authors would like to thank an anonymous referee for making a conjecture about the finite number of interface eigenvalues in the additive interface case. We have proved the corresponding result in the remark following Corollary \ref{C:unique_eval}. The work of T. Dohnal is supported by the Alexander von Humboldt Research Fellowship.

\bibliographystyle{abbrv}
\bibliography{bibliography}

\end{document}